\documentclass{article} % For LaTeX2e
\usepackage{iclr2027_conference,times}

\usepackage{amsmath,amsfonts,bm}

\def\eqref#1{equation~\ref{#1}}
\def\1{\bm{1}}

\DeclareMathAlphabet{\mathsfit}{\encodingdefault}{\sfdefault}{m}{sl}
\SetMathAlphabet{\mathsfit}{bold}{\encodingdefault}{\sfdefault}{bx}{n}

\newcommand{\E}{\mathbb{E}}

\newcommand{\R}{\mathbb{R}}

\DeclareMathOperator*{\argmin}{arg\,min}

\usepackage{enumitem}
\usepackage{algorithm}% http://ctan.org/pkg/algorithms
\usepackage{algpseudocode}
\usepackage{mathtools}
\usepackage{amsthm}
\usepackage{hyperref}
\usepackage{url}

\newcommand{\inner}[2]{\left\langle #1,\, #2 \right\rangle}

\newcommand{\Norm}[1]{\left\|#1\right\|}
\newcommand{\normsq}[1]{\left\|#1\right\|^2}
\usepackage{tikz}
\usetikzlibrary{automata, positioning, arrows}
\usepackage{xcolor}

\def \R {\mathbb{R}}

\def \w {\mathbf{w}}
\def \v {\mathbf{v}}

\def \x {\mathbf{x}}

\def \c {\mathbf{c}}

\def \1 {\mathbf{1}}
\def \z {\mathbf{z}}
\def \s {\mathbf{s}}

\def \y {\mathbf{y}}
\def \u {\mathbf{u}}

\def \L {\mathcal{L}}

\def \I {\mathbf{I}}

\usepackage{titlesec}
\titlespacing*{\section}{0pt}{3pt}{3pt}

\usepackage{hyperref}

\def \B {\mathcalB}

\def \y {\mathbf{y}}
\def \E {\mathbf{E}}

\def \x {\mathbf{x}}

\def \D {\mathcal{D}}
\def \z {\mathbf{z}}
\def \u {\mathbf{u}}

\def \w {\mathbf{w}}

\def \R {\mathbb{R}}

\def \v {\mathbf{v}}

\def \c {\mathbf{c}}

\def \B {\mathcal{B}}
\def \Bb {\mathbb{B}}

\def \s {\mathbf{s}}

\newtheorem{thm}{Theorem}
\newtheorem{prop}{Proposition}
\newtheorem{lem}{Lemma}

\newtheorem{asm}{Assumption}

\title{A Single-Loop, Constant-Batch First-Order Penalty Method for Stochastic Bilevel Optimization}

\author{Xingyu Chen, Ming Yang, Quanqi Hu, Tianbao Yang \thanks{Correspondence to: \texttt{tianbao-yang@tamu.edu}.} \\
Department of Computer Science and Engineering\\
Texas A\&M University, College Station, USA \\
\texttt{\{tamucxy2023, myang, quanqi-hu, tianbao-yang\}@tamu.edu} \\
}

\iclrfinalcopy % Uncomment for camera-ready version, but NOT for submission.
\begin{document}

\maketitle
\lhead{Preprint}

\begin{abstract}
Recent advances in penalty-based methods for stochastic bilevel optimization (SBO) have eliminated the need for second-order derivative oracles. However, for stochastic nonconvex--strongly convex bilevel problems, existing first-order methods typically rely on nested loops and/or large batch sizes for attaining $O(\epsilon^{-6})$ or $O(\epsilon^{-4})$ sample complexity under standard bounded-variance assumption or mean-square smoothness assumption. Achieving these rates with a single-loop penalty method and a constant batch size remains challenging due to a large penalty value needed for an accurate approximation. To address this challenge, we develop a stochastic SIngle-loop COnstant-Batch first-order penalty method (SICO) that combines two complementary ingredients. First, it performs one stochastic-gradient update per-iteration for both the original lower-level and penalized  problems, with a projection that controls the separation between their iterates. Second, it applies an exponential moving average to stabilize the upper-level gradient estimator. We show that this combination achieves $O(\epsilon^{-6})$ sample complexity using only $O(1)$ stochastic-gradient samples per iteration under unbiased, bounded-variance stochastic gradients. Under the additional mean-square smoothness assumption on the lower-level stochastic functions, the same algorithm improves the complexity to $O(\epsilon^{-4})$ also with $O(1)$ batch size. To the best of our knowledge, this is the first work to match the best-known convergence rate for fully first-order SBO methods using a single loop and a constant batch size. This result addresses an open problem posed in the literature.

%Bilevel optimization provides a natural framework for a wide range of machine learning problems, including hyperparameter optimization, data reweighting, and meta-learning. Despite substantial recent progress, obtaining strong convergence guarantees for stochastic nonconvex-strongly convex(NCSC) bilevel optimization without second order oracle remains challenging. Existing methods attaining the best-known \(O(\epsilon^{-6})\) stochastic first-order oracle complexity typically rely on double-loop schemes or large minibatches, while constant-batch single-loop methods have worse complexity. We propose a single-loop penalty method that uses only \(O(1)\) stochastic-gradient samples per iteration. The method jointly tracks the penalized and original lower-level solutions and employs an exponential moving average to avoid dependence on large minibatches when computing stochastic gradient estimator. Under standard unbiased bounded-variance stochastic gradients without stochastic smoothness, our method achieves \(O(\epsilon^{-6})\) sample complexity, matching the best-known rate while retaining a constant-batch single-loop implementation.
\end{abstract}

 \setlength\abovedisplayskip{4pt}
\setlength\belowdisplayskip{4pt}
\setlength{\textfloatsep}{4pt}% Remove \textfloatsep

\section{Introduction}
\label{sec:intro}
In this paper we study \emph{nonconvex--strongly convex (NC--SC) stochastic bilevel optimization (SBO)}:
\begin{equation}
    \min_{\w\in\mathbb R^{d_w}}     F(\w):=f(\w,\y^*(\w)), \qquad
    \y^*(\w):= \arg\min_{\y\in\mathbb R^{d_y}}g(\w,\y), \label{eq:bilevel}
\end{equation}
where the lower level (LL) function $g(\w,\y)\coloneqq \E_{\xi_g\sim\D_g}g(\w,\y, \xi_g)$ is jointly smooth and strongly convex in $\y$ for every $\w$, while the upper level (UL) function $f(\w, \y)\coloneqq \E_{\xi_f\sim\D_f}f(\w,\y, \xi_f)$ is smooth but potentially nonconvex. Both $f$ and $g$ are accessed through stochastic gradients. This formulation captures learning problems in which one optimization task depends on the solution of another, including hyperparameter optimization~\citep{franceschi2018bilevel, bao2021stability}, data hyper-cleaning and reweighting~\citep{gao2022self}, and meta-learning with strongly convex inner problems
\citep{ji2021bilevel, kwon2023fully}.

A substantial body of work has developed algorithms and convergence guarantees for this problem class. Representative contributions include early approximation methods~\citep{ghadimi2018approximation}, two-timescale stochastic approximation~\citep{doi:10.1137/20M1387341}, improved hypergradient estimation and alternating stochastic updates~\citep{ji2021bilevel,chen2021closing}, momentum-based single-loop methods~\citep{10.1007/s10994-024-06650-8}, and variance-reduced algorithms~\citep{10.5555/3540261.3541308,10.5555/3540261.3542577,DBLP:journals/corr/abs-2105-02266}. These approaches estimate the \emph{hypergradient} $\nabla F(\w)$ by differentiating through the lower-level solution map. This typically requires second-order information, accessed through Hessian- or Jacobian-vector products (HVPs/JVPs) or approximate inverse-Hessian computations.

%A substantial body of work has developed algorithms and convergence guarantees for this problem class \citep{ghadimi2018approximation,ji2021bilevel}. A common approach estimates the \emph{hypergradient} $\nabla F(\w)$ by differentiating through the lower-level solution. This requires second-order information, typically through Hessian- or Jacobian-vector products(HVP/JVP) or approximate inverse-Hessian computations. Although these approaches need not form a full Hessian matrix, their second-order evaluations can still incur substantial computational and memory costs. This motivates \emph{fully first-order} methods, which use only ordinary stochastic gradients of $f$ and $g$ and avoid second-order oracle access altogether.

To reduce the computational and memory costs associated with second-order derivative computations, recent work has developed \emph{fully first-order} methods, which use only ordinary stochastic gradients of $f$ and $g$ and avoid second-order oracle access altogether. A prominent approach is to construct a penalty-based reformulation of the bilevel problem that can be optimized using only first-order information~\citep{3600270.3601524,kwon2023fully,pmlr-v202-shen23c}. In particular, problem~(\ref{eq:bilevel}) admits the equivalent constraint reformulation:
\begin{equation}
    \min_{\w,\y} f(\w,\y)
    \qquad\text{subject to}\qquad
    g(\w,\y)-\min_{\z}g(\w,\z)\le 0.
    \label{eq:value-constrained}
\end{equation}
Penalizing the constraint violation yields a surrogate problem:
\begin{equation}
    \min_{\w}\mathcal L^*_\lambda(\w),
    \qquad
    \mathcal L^*_\lambda(\w)
    :=\min_{\y}\left\{
    f(\w,\y)+\lambda\left[g(\w,\y)-\min_{\z}g(\w,\z)\right]
    \right\},
    \quad \lambda>0.
    \label{prob:penalty}
\end{equation}
% The key advantage is that the gradient of the surrogate objective has the
% first-order representation
% \begin{equation}
% \begin{aligned}\label{eq:penalty-gradient}
% \nabla\L^*_\lambda(\w) &= \nabla_1 \L_\lambda(\w, \y_\lambda^*(\w)) + \nabla\y_\lambda^*(\w)\underbrace{\nabla_2 \L_\lambda(\w, \y_\lambda^*(\w))}_{=\mathbf{0}}\\
% &=\nabla_1 f\big(\w,\y^*_\lambda(\w)\big)+\lambda\left[
% \nabla_1 g\big(\w,\y^*_\lambda(\w)\big) - \nabla_1 g\big(\w,\y^*(\w)\big)\right].
% \end{aligned}
% \end{equation}
Under standard regularity conditions and for sufficiently large $\lambda$, the gradient of the surrogate objective approximates the bilevel hypergradient with error
$\|\nabla\mathcal L_\lambda^*(\w)-\nabla F(\w)\|_2=O(\lambda^{-1})$~\citep{kwon2023fully}.
Consequently, choosing $\lambda=\Theta(\epsilon^{-1})$ ensures that an $O(\epsilon)$-stationary point of $\mathcal L_\lambda^*$ is also an $O(\epsilon)$-stationary point of $F$.

Despite this first-order representation, attaining sharp sample complexity with simple updates remains challenging. Existing fully first-order methods achieve $O(\epsilon^{-6})$ sample complexity under bounded-variance stochastic gradients and $O(\epsilon^{-4})$ under additional mean-square smoothness. However, the methods attaining these rates rely on nested lower-level solvers~\citep{ kwon2024complexity,JMLR:v26:23-1104,pmlr-v202-shen23c}, $\epsilon$-dependent large batch sizes~\citep{kwon2024complexity,JMLR:v26:23-1104,gu2026sgha,pmlr-v202-shen23c}, or both. Meanwhile, existing methods that combine a single loop with constant-size batches have weaker sample complexity guarantees~\citep{kwon2023fully,yang2026compositional}. This gap motivates the following open question~\citep{wright2025gradient,yang2026compositional}:
\begin{center}
\emph{Can we design a fully first-order method for NC-SC SBO that uses a single loop and constant-size batches while matching the best-known sample complexity under standard stochastic assumptions?}
\end{center}

We answer this question affirmatively with a simple algorithmic design supported by refined convergence analyses. Our method maintains $\y_t$ and $\z_t$ to track the minimizers of the penalized problem over $\y$ and original lower-level problem over $\z$ in~(\ref{prob:penalty}), respectively, taking one stochastic-gradient step for each at every iteration given $\w_t$. Inspired by the projected-coupling idea of \citet{kwon2024complexity}, we  project the provisional $\z$-iterate onto a ball of radius $O(\lambda^{-1})$ centered at the updated $\y$-iterate. This projection constrains the separation between the two iterates rather than their individual locations. We combine it with a standard exponential moving average of the stochastic gradient estimator for the upper-level variable $\w$. The projection controls the relative lower-level tracking error, while the moving average controls the upper-level gradient estimation error with a constant batch size. We refer to the proposed algorithm as SICO. 

By establishing refined contraction and sensitivity bounds for coupled lower-level tracking, we prove that SICO achieves $O(\epsilon^{-6})$ sample complexity under unbiased, bounded-variance stochastic gradients and $O(\epsilon^{-4})$ under the additional mean-square smoothness assumption on the lower-level stochastic gradients, using only $O(1)$ stochastic-gradient samples per iteration in both settings.  In addition, our results also achieve  the  best known dependence on the condition number (cf. Table~\ref{tab:baseline}).

We evaluate the proposed method on bilevel problems arising in data hyper-cleaning and regularization selection, demonstrating its effectiveness compared with existing fully first-order baselines.

% \begin{enumerate}[leftmargin=*]
%     \item We develop SICO, a single-loop, constant-batch, fully first-order penalty method that not only matches the best-known sample complexity in both the bounded-variance and mean-square smoothness settings, but also improves the dependence on the condition number (cf. Table~\ref{tab:baseline}).

%     \item We establish refined contraction and sensitivity bounds for coupled lower-level tracking, extending the analysis of projected coupling to the bounded-variance setting.
% \end{enumerate}

% We prove that the resulting algorithm, dubbed SICO, achieves $O(\epsilon^{-6})$ sample complexity under unbiased, bounded-variance stochastic gradients and $O(\epsilon^{-4})$ under the additional mean-square smoothness assumption on the lower-level stochastic gradients, using only $O(1)$ stochastic-gradient samples per iteration in both settings. Our main contributions are summarized as follows.

% \begin{enumerate}[leftmargin=*]
%     \item We develop SICO, a single-loop, constant-batch, fully first-order penalty method that not only matches the best-known sample complexity in both the bounded-variance and mean-square smoothness settings, but also improves the dependence on the condition number (cf. Table~\ref{tab:baseline}).

%     \item We establish refined contraction and sensitivity bounds for coupled lower-level tracking, extending the analysis of projected coupling to the bounded-variance setting.
% \end{enumerate}

\begin{table}[t]
    \caption{Comparison of fully first-order methods for bilevel
    optimization. The second column indicates the mean-square smoothness (MSS) of the stochastic gradients (cf. Assumption~\ref{asm:mss}).
    % N/A denotes not applicable. For the deterministic method,    complexity counts exact gradient evaluations.
    }
    \label{tab:baseline}
    \centering
    \begin{tabular}{lcccc}
        \hline
        Algorithm
        & \shortstack{MSS}
        & Loop
        & Batch size
        & Sample complexity \\
        \hline
        F2SA~\citep{kwon2023fully}
        & No & Single & $O(1)$
        & $O(\operatorname{poly}(\kappa)\epsilon^{-7})$ \\

        \citet{yang2026compositional}& No & Single & $O(1)$
        & $O(\kappa^{15}\epsilon^{-7})$ \\
        \citet{kwon2024complexity}
        & No & Double & $O(\epsilon^{-4})$
        & $O(\operatorname{poly}(\kappa)\epsilon^{-6})$ \\

        F2BSA~\citep{JMLR:v26:23-1104}
        & No & Double & $O(\epsilon^{-4})$
        & $\widetilde{O}(\kappa^{12}\epsilon^{-6})$ \\

        F3BSA~\citep{JMLR:v26:23-1104}
        & No & Double & $O(\kappa)$
        & $\widetilde{O}(\kappa^{11}\epsilon^{-6})$ \\

        Stoc-SGHA~\citep{gu2026sgha} %Theorem 9
        & No & Single & $O(\epsilon^{-4})$
        & $O(\kappa^{17}\epsilon^{-6})$ \\

       \textbf{SICO (Theorem~\ref{thm:std})}
    & \textbf{No}
    & \textbf{Single}
    & {\boldmath $O(1)$}
    & {\boldmath $O(\kappa^{11}\epsilon^{-6})$} \\
        \hline
        F3SA~\citep{kwon2023fully}
        & Yes & Single & $O(1)$
        & $O(\operatorname{poly}(\kappa)\epsilon^{-5})$ \\
        
        \citet{kwon2024complexity}
        & Yes & Double & $O(\epsilon^{-2})$
        & $O(\operatorname{poly}(\kappa)\epsilon^{-4})$ \\
        
        Stoc-SGHA~\citep{gu2026sgha} %Theorem 11
        & Yes & Single & $O(\epsilon^{-2})$
        & $O(\kappa^{11}\epsilon^{-4})$ \\
        
        \textbf{SICO (Theorem~\ref{thm:mss})}
    & \textbf{Yes}
    & \textbf{Single}
    & {\boldmath $O(1)$}
    & {\boldmath $O(\kappa^{9}\epsilon^{-4})$} \\

        %Ours (deterministic)
        %& N/A & Single & N/A
        %& $O(\epsilon^{-2})$ \\
        \hline
    \end{tabular}
%\vspace*{-0.2in}
\end{table}
\section{Related Work}

%\noindent{\bf Historical development of penalty methods.}
Penalty methods for bilevel optimization build on classical value-function reformulations and associated exact-penalty analyses~\citep{Ye1995OptimalityCF}. These formulations express lower-level optimality through a constraint on the lower-level optimality gap. More recent developments include~\citep{3600270.3601524,pmlr-v202-shen23c,doi:10.1137/23M1566753}. Building on this broader literature, recent research on NC--SC SBO has focused on controlling the approximation error and stochastic noise introduced by penalization.

%\noindent{\bf Fully first-order penalty methods for NC--SC SBO.}
\citet{kwon2023fully} introduced F2SA, a fully first-order penalty method that avoids Jacobian and Hessian evaluations. It uses simple stochastic gradient updates for all variables.  Its established sample complexity is $O(\epsilon^{-7})$ in the fully stochastic setting. Their F3SA variant incorporates STORM-style gradient estimators for both the upper- and lower-level variables, improving the complexity to $O(\epsilon^{-5})$ under additional mean-square smoothness for both the upper and lower level stochastic gradients while using a single loop and constant-size batches.

\citet{kwon2024complexity} established two improved complexity bounds: $O(\epsilon^{-6})$ under bounded-variance stochastic gradients and $O(\epsilon^{-4})$ under additional mean-square smoothness. Their method uses a double-loop structure with batch sizes of order $\epsilon^{-4}$ and $\epsilon^{-2}$ in the two settings, respectively. The analysis exploits the proximity of the original and penalized lower-level solutions through projection and, under mean-square smoothness. These ideas motivate our algorithm, while our refined tracking analysis enables single-loop updates with constant-size batches.

\citet{JMLR:v26:23-1104} showed that the smoothness constant of the reduced penalty objective $\L_\lambda^*$ can be bounded independently of $\lambda$. Their F2BSA method combines inner solvers with large batch sizes to attain $\widetilde O(\kappa^{12}\epsilon^{-6})$ sample complexity in the fully stochastic setting. Their F3BSA variant uses multilevel Monte Carlo and single-sample stochastic updates to attain $\widetilde O(\kappa^{11}\epsilon^{-6})$ complexity. However, constructing its gradient estimator still requires nested stochastic solvers and a batch size in the order of $O(\kappa)$. 

Using an equivalent min--max reformulation of the penalty problem~(\ref{prob:penalty}), \citet{yang2026compositional} developed a single-loop stochastic momentum gradient descent--ascent method with constant-size batches. Its established sample complexity under bounded variance is $O(\epsilon^{-7})$, leaving a gap relative to the $O(\epsilon^{-6})$ dependence achieved by nested-loop methods.

%\noindent{\bf Concurrent single-loop methods using finite differences.}
In concurrent work, \citet{gu2026sgha} proposed SGHA and its stochastic counterpart, Stoc-SGHA, based on a regularized Lagrangian formulation of the lower-level stationarity constraints. These methods approximate Hessian--vector products using finite differences of first-order gradients. Stoc-SGHA achieves expected oracle complexities of $O(\kappa^{17}\epsilon^{-6})$ under bounded variance and $O(\kappa^{11}\epsilon^{-4})$ under lower-level mean-square smoothness. However, its stated in-expectation guarantees additionally assume that the iterates are almost surely bounded and use batch sizes of order $\epsilon^{-4}$ and $\epsilon^{-2}$, respectively. Thus, its single-loop structure still entails an $\epsilon$-dependent per-iteration sampling cost. In contrast, SICO operates directly on stochastic gradients and does not require finite-difference approximations or an additional almost-sure bounded-iterate assumption and has sharper dependence on the condition mumber. Table~\ref{tab:baseline} summarizes these comparisons.

\section{Preliminaries}
\noindent{\bf Notation.} We use $\nabla_1f(\w, \y)$ and $\nabla_2 f(\w,\y)$ to denote gradients with respect to the first and second arguments, respectively, and $\|\cdot\|$ to denote the Euclidean norm. Let $\Pi_{\mathbb{B}(\c,r)}(\u)$ denote the Euclidean projection of $\u$ onto  a closed Euclidean ball $\mathbb{B}(\c,r)=\{\u:\|\u-\c\|\le r\}$. We write $\E_t[\cdot]$ for the conditional expectation given the history prior to drawing the random samples at iteration $t$. A possibly random point $\w$ is called $\epsilon$-stationary if $\E[\|\nabla F(\w)\|^2]\le\epsilon^2$. A function $f$ is $C$-Lipschitz continuous if $|f(\x)-f(\y)|\le C\|\x-\y\|$ for all $\x,\y$. It is $L$-smooth if it is differentiable and $\|\nabla f(\x)-\nabla f(\y)\|\le L\|\x-\y\|$ for all $\x,\y$. A differentiable function $g$ is $\mu_g$-strongly convex if $g(\y)\ge g(\x)+\nabla g(\x)^\top(\y-\x)+\frac{\mu_g}{2}\|\y-\x\|^2$ for all $\x,\y$. A mapping $H(\s):\R^d\rightarrow \R^d$ is $\mu$-strongly monotone if $(H(\s)-H(\s'))^{\top}(\s - \s')\geq \mu\|\s - \s'\|^2$. 

\subsection{Assumption}
We make the following standard assumptions regarding the problem~(\ref{eq:bilevel}). 
\begin{asm}\label{asm:func}Regarding $f, g$ and $F$, we assume:\\
    $\bullet$ $f(\w, \y)$ is $C_f$-Lipschitz in $\y$ and jointly $L_f$-smooth in terms of $(\w, \y)$\\
    $\bullet$ $g(\w, \y)$ is jointly $L_g$-smooth and $\mu_g$-strongly convex in terms of $\y$ for any $\w$\\
    $\bullet$ $f$ is twice continuously differentiable, and $\nabla^2 g(\w, \y)$ is jointly $Q_g$-Lipschitz continuous.\\
    $\bullet$ $F(\w)$ is lower bounded, i.e., $\inf_{\w\in\R^{d_w}} F(\w)> -\infty$
\end{asm}
\begin{asm}\label{asm:oracle}There exist stochastic gradients of $f$ and $g$ that are unbiased and have bounded variance $\sigma_f^2, \sigma_g^2$, respectively, i.e.
    \begin{align}
        \E_{\xi_f} [\nabla f(\w, \y, \xi_f)] = \nabla f(\w, \y)&, \quad\E_{\xi_g} [\nabla g(\w, \y, \xi_g)] = \nabla g(\w, \y)\nonumber\\
        \E_{\xi_f}\normsq{\nabla f(\w, \y, \xi_f) - \nabla f(\w, \y)} \leq \sigma_f^2&, \quad \E_{\xi_g}\normsq{\nabla g(\w, \y, \xi_g) - \nabla g(\w, \y)} \leq \sigma_g^2.
    \end{align}
\end{asm}
For deriving improved convergence rate, we additionally consider the following mean-square smoothness (MSS) assumption regarding stochastic gradient of $g$.
\begin{asm}\label{asm:mss}  $g$ is mean-square smooth in $\y$, i.e.
\begin{align}
    \E_{\xi_g}\normsq{\nabla g(\w, \y_1, \xi_g) - \nabla g(\w, \y_2, \xi_g)} \leq \tilde L_g^2 \normsq{\y_1 - \y_2}.
\end{align}
\end{asm}
We follow the tradition from~\citep{JMLR:v26:23-1104} and define $\kappa\coloneqq \bar{\ell}/\mu_g$ where $\bar{\ell} \coloneqq \max\{C_f, L_f, L_g, Q_g\}$. When lower-level objective is also $\tilde L_g$-MSS, we slightly abuse the notation and define $\bar{\ell} \coloneqq \max\{C_f, L_f, L_g, \tilde L_g, Q_g\}$.

\subsection{Penalty Surrogate}
Recall the penalized problem in~(\ref{prob:penalty}). For any $\w$, we denote the unique minimizer of the strongly convex lower-level problem by
\begin{align}
\y^*(\w)\coloneqq\argmin_{\z}g(\w,\z). %,\qquad g^*(\w)\coloneqq\min_{\y}g(\w,\y)=g(\w,\y^*(\w)).
\end{align}
For a penalty parameter $\lambda>0$, define
\begin{align}
\L_\lambda(\w,\y)\coloneqq f(\w,\y)+\lambda\left(g(\w,\y)-\min_\z g(\w, \z)\right). 
\end{align}
When $\lambda\geq 2L_f/\mu_g$, $\lambda^{-1}f(\w,\cdot)+g(\w,\cdot)$ is $(\mu_g - \frac{L_f}{\lambda})$-strongly convex and we define  $\y_\lambda^*(\w)$ as the unique minimizer of the penalized function $\L_\lambda(\w,\y)$ with respect to the lower-level variable given $\w$, i.e.,
\begin{align}
\y_\lambda^*(\w)\coloneqq\argmin_{\y}\L_\lambda(\w,\y)
=\argmin_{\y}\left\{\lambda^{-1}f(\w,\y)+g(\w,\y)\right\}.
\label{def:penal_ll_sol}
\end{align}
 The penalized objective introduced in~(\ref{prob:penalty}) can then be written as
\begin{align}
\L_\lambda^*(\w)\coloneqq\L_\lambda(\w,\y_\lambda^*(\w))
=f(\w,\y_\lambda^*(\w))+\lambda\left(g(\w,\y_\lambda^*(\w))-g(\w,\y^*(\w))\right).
% \label{eq}
\end{align}
 Moreover, 
% \begin{align}
% \nabla\L_\lambda^*(\w)=\nabla_1f(\w,\y_\lambda^*(\w))+\lambda\left(\nabla_1 g(\w,\y_\lambda^*(\w))-\nabla_1g(\w,\y^*(\w))\right),
% \label{eq}
% \end{align}

\begin{equation}
\begin{aligned}\label{eq:penalty-gradient}
\nabla\L^*_\lambda(\w) &= \nabla_1 \L_\lambda(\w, \y_\lambda^*(\w)) + \nabla\y_\lambda^*(\w)\underbrace{\nabla_2 \L_\lambda(\w, \y_\lambda^*(\w))}_{=\mathbf{0}}\\
&=\nabla_1 f\big(\w,\y^*_\lambda(\w)\big)+\lambda\left[
\nabla_1 g\big(\w,\y^*_\lambda(\w)\big) - \nabla_1 g\big(\w,\y^*(\w)\big)\right],
\end{aligned}
\end{equation}
which involves only first-order derivatives once $\y_\lambda^*(\w)$ and $\y^*(\w)$ are available as described in the previous section. %This representation motivates the two lower-level estimators used in Section~\ref{sec}.

We next present the approximation and smoothness properties of the penalty surrogate as established in previous work. %These properties justify optimizing \(\mathcal L_\lambda^*\) in place of \(F\) and will be used in the convergence analysis.
\begin{lem}[Lemma 4.1 of~\citet{JMLR:v26:23-1104}]
\label{lem:penal_property}
Under Assumption~\ref{asm:func}, for $\lambda\ge 2L_f/\mu_g$, it holds for $\forall \w\in\R^{d_w}$that
\begin{align}
% \Norm{\y^*_\lambda(\w)-\y^*(\w)} &\leq \frac{C_f}{\mu_g\lambda} \eqqcolon r_\lambda, \label{eq}\\
&|\L^*_\lambda(\w) - F(\w)| = O\left(\frac{\bar{\ell}\kappa^2}{\lambda}\right),
\quad \Norm{\nabla\L^*_\lambda(\w)-\nabla F(\w)} = O\left(\frac{\bar{\ell}\kappa^3}{\lambda}\right),
\label{eq:obj_grad_surrogate_approx}\\
&\Norm{\nabla\L^*_\lambda(\w_1) - \nabla\L^*_\lambda(\w_2)} \leq L_*\Norm{\w_1-\w_2}, \text{ where } L_*=O(\bar{\ell}\kappa^3).
\label{eq:sm_surrogate}
\end{align}
\end{lem}

We make two remarks regarding this lemma. First, choosing $\lambda=\Omega(\bar{\ell}\kappa^3/\epsilon)$ ensures that the gradient of the penalty objective approximates the true hypergradient to within $O(\epsilon)$. Consequently, any $O(\epsilon)$-stationary point of $\L_\lambda^*$ is also an $O(\epsilon)$-stationary point of $F$. Second, although $\lambda$ must increase as $\epsilon$ decreases, the smoothness constant $L_*$ of the reduced penalty objective does \emph{not} grow with $\lambda$. Thus, increasing the penalty parameter does not, by itself, require a smaller outer step size on smoothness grounds. The main challenge posed by a large $\lambda$ lies in stochastic gradient estimation, where the penalty factor can amplify the estimation error.

\section{A penalty-based first-order single loop method}
In this section, we present our stochastic single-loop constant-batch first-order method. We denote by $\B_{fw}^t, \B_{fy}^t$ independent batches for evaluating stochastic gradients of $f(\w, \y)$ at the $t$-th iteration. Similarly, denote  by $\B_{gw}^t, \B_{gy}^t$ independent batches for evaluating stochastic gradients of $g(\w, \y)$ at the $t$-th iteration. We assume $|\B_{fw}^t|=|\B_{fy}^t| =|\B_{gw}^t| = |\B_{gy}^t|=B=O(1)$. 

\subsection{The Algorithm}
We begin by motivating the proposed algorithm. By equation~(\ref{eq:penalty-gradient}), the gradient of $\L_\lambda^*(\w)$ has the compositional form $\mathcal M(\w,\y^*(\w),\y_\lambda^*(\w))$, where $\mathcal M$ is Lipschitz continuous and $\y^*(\w)$ and $\y_\lambda^*(\w)$ are solutions to strongly convex optimization problems. This structure fits the compositional gradient framework discussed in~\cite[Section~4.5]{yang2026compositional}, which provides a general scheme for estimating such gradients.

To estimate the gradient in equation~(\ref{eq:penalty-gradient}), we maintain two sequences, $\y_{t+1}$ and $\z_{t+1}$, to track $\y_\lambda^*(\w_t)$ and $\y^*(\w_t)$, respectively. Substituting these iterates for the corresponding optimal solutions and replacing the full gradients with stochastic estimates yields
\begin{align}
    \hat G_t^w
    \coloneqq \nabla_1 f(\w_t,\y_{t+1},\B_{fw}^t)
    +\lambda\left(
    \nabla_1 g(\w_t,\y_{t+1},\B_{gw}^t)
    -\nabla_1 g(\w_t,\z_{t+1},\B_{gw}^t)
    \right).
    \label{eq:upper_grad}
\end{align}
Because the iterates $\y_{t+1}, \z_{t+1}$ only approximate their respective targets, this estimator is generally biased:
$G_t^w\coloneqq\E_{\B^t_{fw}, \B^t_{gw}}[\hat G_t^w]\neq\nabla\L_\lambda^*(\w_t)$.
To reduce stochastic fluctuations, we use an exponential moving average with $\beta\in(0,1)$,
\begin{align}
    \v_{t+1}\coloneqq(1-\beta)\v_t+\beta\hat G_t^w,
\end{align}
and update the primal variable by
\begin{align}
    \w_{t+1}\coloneqq\w_t-\eta\v_{t+1}.
    \label{eq:mmt_update}
\end{align}

Controlling the {\bf tracking-induced squared bias} $\|G_t^w-\nabla\L_\lambda^*(\w_t)\|^2$, however, requires suitable updates for the inner iterates. 
A natural choice, following~\citep{yang2026compositional}, is to perform a single SGD step for each inner problem at every iteration. However, when $\lambda$ scales as $1/\epsilon$, this approach yields a complexity of $O(\epsilon^{-7})$. The bottleneck is controlling the squared bias $\|G_t^w-\nabla\L_\lambda^*(\w_t)\|^2$ despite the large penalty parameter. Specifically, for $\lambda\geq 1$, a direct Lipschitz bound gives \begin{align*} \|G_t^w-\nabla\L_\lambda^*(\w_t)\|^2 \leq O(\lambda^2)\left( \|\y_{t+1}-\y_\lambda^*(\w_t)\|^2 +\|\z_{t+1}-\y^*(\w_t)\|^2 \right). \end{align*} This bound amplifies both tracking errors by a factor of $O(\lambda^2)$, requiring increasingly accurate inner solutions as $\epsilon$ decreases and thereby worsening the overall complexity.

\paragraph{Projected lower-level tracking.}
Inspired by~\citet{kwon2024complexity}, our algorithm exploits the proximity of $\y_\lambda^*(\w)$ and $\y^*(\w)$. Specifically, for every $\w$, we have (see Appendix~\ref{sec:lem})
\begin{align}
    \Norm{\y_\lambda^*(\w)-\y^*(\w)}
    \leq r_\lambda\coloneqq \frac{C_f}{\mu_g\lambda}.
    \label{eq:lower_sol_gap}
\end{align}
We use this bound to couple the updates of $\y$ and $\z$. After taking an SGD step for each variable, we project the updated $\z$ onto the ball of radius $r_\lambda$ centered at the updated $\y$. The resulting updates are
\begin{align}
    &\hat G_t^y
    = \lambda^{-1}\nabla_2 f(\w_t,\y_t,\B_{fy}^t)
    +\nabla_2 g(\w_t,\y_t,\B_{gy}^t), \qquad
    \hat G_t^z
    = \nabla_2 g(\w_t,\z_t,\B_{gy}^t),
    \label{eq:lower_grad}\\
    &\y_{t+1} = \y_t-\gamma\hat G_t^y,\qquad
    \bar{\z}_t = \z_t-\gamma\hat G_t^z,
    \label{eq:lower_update}\\
    &\z_{t+1}
    = \Pi_{\Bb(\y_{t+1},r_\lambda)}(\bar{\z}_t).
    \label{eq:z_proj}
\end{align}

Combining these projected lower-level updates with the momentum update of $\w$ yields our algorithm, SICO, presented in Algorithm~\ref{alg:s2f2bo}. Compared with~\citep{kwon2024complexity}, our main algorithmic modification is to apply this projection at every iteration and combine it with momentum-based gradient estimation. This yields a single-loop method that requires neither inner loops nor large batches supported by our novel analysis.

We next present convergence guarantees for SICO with and without the mean-square smoothness assumption on $g$, and then highlight the novelty of our analysis.

\begin{algorithm}[t]
    \caption{Stochastic Single-loop Fully First-order Penalty Method (SICO) for SBO}
    \label{alg:s2f2bo}
    \begin{algorithmic}[1]
        \State \textbf{Input:} $T$, $\lambda$, $\w_0=\w_{-1}$, $\y_0 = \y_{-1}$, $\z_0 = \Pi_{\Bb(\y_0,r_\lambda)}(\z_{-1})$, $\v_0 = \textbf{0}$, $\eta$, $\gamma$, $\beta$, $r_\lambda=C_f/(\mu_g\lambda)$. 
        % \State Set $\y_0 = \y_\text{raw}, \z_0 = \Pi_{\Bb(\y_0,r_\lambda)}(\z_\text{raw})$
        \For{$t=0,1,\cdots,T-1$}
            \State Sample independent constant-size batches $\B_{fy}^t,\B_{gy}^t,\B_{fw}^t,\B_{gw}^t$% from $\D_f,\D_g$, correspondingly.
            \State Compute $\hat G_t^y=\lambda^{-1}\nabla_2 f(\w_t,\y_t,\B_{fy}^t)+\nabla_2 g(\w_t,\y_t,\B_{gy}^t)$ and $\hat G_t^z=\nabla_2 g(\w_t,\z_t,\B_{gy}^t)$.
            \State Update $\y_{t+1}=\y_t-\gamma\hat G_t^y$ and $\bar{\z}_t=\z_t-\gamma\hat G_t^z$.
            \State Project $\z_{t+1}=\Pi_{\Bb(\y_{t+1},r_\lambda)}(\bar{\z}_t)$.
            \State Compute $\hat G_t^w=\nabla_1 f(\w_t,\y_{t+1},\B_{fw}^t)+\lambda\left(\nabla_1 g(\w_t,\y_{t+1},\B_{gw}^t)-\nabla_1 g(\w_t,\z_{t+1},\B_{gw}^t)\right)$.
            \State Update $\v_{t+1}=(1-\beta)\v_t+\beta\hat G_t^w$ and $\w_{t+1}=\w_t-\eta\v_{t+1}$.
        \EndFor
        \State \textbf{Output:} $\w_\tau$ with $\tau$ uniformly sampled from $\{0,1,\cdots,T-1\}$.
    \end{algorithmic}
\end{algorithm}

\subsection{Convergence under standard bounded variance}
We first consider Assumptions~\ref{asm:func} and~\ref{asm:oracle}, without imposing mean-square smoothness on the stochastic gradient of $g$. The following theorem shows that Algorithm~\ref{alg:s2f2bo} matches the $O(\epsilon^{-6})$ dependence of the best known first order methods while using a single loop and only a constant batch size.
\begin{thm}\label{thm:std}
Under Assumptions~\ref{asm:func} and~\ref{asm:oracle}, let the batch size be $B=O(1)$. With a warm-up stage for $\y$ such that the initialization satisfies the condition that $\E[\|\y_0 - \y^*_\lambda(\w_0)\|^2]\leq O(1/\kappa)$, by setting
\[
    \lambda=\Theta\left(\frac{\kappa^3}{\epsilon}\right),\qquad
    \beta=\Theta\left(\frac{\epsilon^4}{\kappa^6}\right),\qquad
    \gamma=\Theta\left(\frac{\epsilon^4}{\kappa^7}\right),\qquad
    \eta=\Theta\left(\frac{\epsilon^4}{\kappa^{11}}\right),
\]
Algorithm~\ref{alg:s2f2bo} returns $\w_\tau$ satisfying
\[
    \E\normsq{\nabla F(\w_\tau)}\leq\epsilon^2
\]
within $T=O(\kappa^{11}\epsilon^{-6})$ stochastic first-order oracle calls, where $\tau$ is uniformly sampled from $\{0, 1, \cdots, T-1\}$.
\end{thm}

\textbf{Remark:} %The goal of the warm-up stage for $\y$ is to get a good initialization such that the initial error of $\y$ is small. 
Note that the involved optimization problem in~(\ref{def:penal_ll_sol}) that define $\y_\lambda^*(\w_0)$ is strongly convex, hence the warm-up stage is cheap. For example, running a simple SGD for $T_0=O(\kappa^3)$-iterations for solving $\min_{\y}\L_\lambda(\w_0,\y)$ will find a solution $\y_0$ such that $\E[\|\y_0 - \y^*_\lambda(\w_0)\|^2]\leq O(1/(\mu^2_g T_0))=O(1/\kappa)$.

% \paragraph{Refined analysis of the coupled lower-level iterates.}
% The key is to analyze the iterate gap
% $\mathbf s_t:=\y_t-\z_t$ relative to
% $\mathbf s^*(\w):=\y_\lambda^*(\w)-\y^*(\w)$, rather than bound the two lower-level errors separately. Two refinements make this possible. First, a direct application of strong convexity does not establish contraction of the gap: the two iterates are updated using gradients of different objectives. We instead derive a coupled contraction inequality that explicitly retains the error in tracking $\y_\lambda^*$, with a coefficient of order $\lambda^{-2}$. This additional term can be controlled jointly with the recursion of the $\y_\lambda^*$ tracking error. Second, we establish the $O(\lambda^{-1})$-lipschitzness of $\s^*(\cdot)$(cf. Lemma~\ref{lem:s_lip})
% \begin{equation}
%     \|\s^*(\w)- \s^*(\w')\|^2
%     \le \frac{C_s^2}{\lambda^2}\|\w-\w'\|^2,
%     \label{eq:s_lip_intro}
% \end{equation}
% where $C_s$ is independent of $\lambda$. Compared with~\citet{kwon2024complexity}, this bound removes an additive $O(\lambda^{-4})$ residual that would obstruct the desired single-loop analysis. The above 2 refinements together yield a one-step recursion for $\s$ with/without mean-square smoothness. Combining these refinements with the momentum-estimation recursion and descent of $\mathcal L_\lambda^*$ allows us to attain the $O(\epsilon^{-6})$ rate without double loop or large minibatches.

\subsection{Improved rate under mean-square smoothness} %\label{sec:main_mss}
We next establish an improved rate assuming the mean-square smoothness of $g$ as in Assumption~\ref{asm:mss}. The intuition behind the improvement is that under mean-square smoothness of $g$, {\bf the variance of the estimator} $\hat G_t^w$ will become independent of $\lambda$. 
% In this case, because the two evaluations of $g$ in~(\ref{eq:upper_grad}) use the same samples and the projection guarantees $\Norm{\y_{t+1}-\z_{t+1}}\leq r_\lambda=O(1/\lambda)$, the $\lambda^2$ variance amplification can be eliminated. In particular, Lemma~\ref{lem:grad_w_var} gives
% \begin{align}
%     \operatorname{Var}(\hat G_t^w)\leq\frac{1}{B}\left(2\sigma_f^2+\frac{8\tilde L_g^2C_f^2}{\mu_g^2}\right), \label{eq:upper_var_mss_method}
% \end{align}
% which is independent of $\lambda$. 
As a result, the same algorithm therefore provides a stronger benefit under Assumption~\ref{asm:mss}, leading to the improved rate below.
\begin{thm}\label{thm:mss}
Under Assumptions~\ref{asm:func},~\ref{asm:oracle} and~\ref{asm:mss}, let the batch size be $B=O(1)$. With the same warm-up initialization and by setting
\[
    \lambda=\Theta\left(\frac{\kappa^3}{\epsilon}\right),\qquad
    \beta=\Theta\left(\frac{\epsilon^2}{\kappa^2}\right),\qquad
    \gamma=\Theta\left(\frac{\epsilon^2}{\kappa^5}\right),\qquad
    \eta=\Theta\left(\frac{\epsilon^2}{\kappa^9}\right),
\]
Algorithm~\ref{alg:s2f2bo} returns $\w_\tau$ satisfying
\[
    \E\normsq{\nabla F(\w_\tau)}\leq\epsilon^2
\]
within $T=O(\kappa^9\epsilon^{-4})$ stochastic first-order oracle calls, where $\tau$ is uniformly sampled from $\{0, 1, \cdots, T-1\}$.
\end{thm}
{\bf Remark:} To the best of our knowledge, this rate also has the best dependence on $\kappa$ compared with that of other fully first-order methods.

\subsection{Sketch of Analysis}
We outline the convergence analysis and highlight its novel elements. Let $\E_t$ denote expectation over the randomness at iteration $t$, conditional on $\w_t,\y_t,\z_t$. Define
\begin{align*}
    \s_t &\coloneqq \y_t-\z_t,\; \s^*(\w_t) \coloneqq \y_{\lambda}^*(\w_t) - \y^*(\w_t),\\  Y_t \coloneqq \normsq{\y_{t+1}-\y_\lambda^*(\w_t)},\; &S_t \coloneqq \normsq{\s_{t+1}-\s^*(\w_t)},\; V_t \coloneqq \normsq{\v_{t+1}-\nabla\L_\lambda^*(\w_t)}.
\end{align*}

The smoothness of $\L_\lambda^*$ in~(\ref{eq:sm_surrogate}) yields the following descent inequality for $\eta\leq 1/(2L_*)$ (see Lemma~\ref{lem:descent}):
\begin{align}\label{eq:noncvx_starter_0}
    \E_t\L_\lambda^*(\w_{t+1})
    \leq \L_\lambda^*(\w_t)
    +\frac{\eta}{2}\E_t V_t
    -\frac{\eta}{4}\E_t\normsq{\v_{t+1}}
    -\frac{\eta}{2}\normsq{\nabla\L_\lambda^*(\w_t)}.
\end{align}
The moving-average update further gives the following recursion for the gradient estimation error (see Lemma~\ref{lem:mmt_recur}):
\begin{align*}
    \E_{t+1} V_{t+1}
    \leq {}&(1-\beta)V_t
    +\frac{2L_*^2\eta^2}{\beta}\normsq{\v_{t+1}}+4\beta\normsq{G^w_{t+1}-\nabla\L_\lambda^*(\w_{t+1})}
    +\beta^2\E_{t+1}\normsq{\hat G^w_{t+1}-G^w_{t+1}}.
\end{align*}
The next step is to control the squared bias
$\normsq{G^w_{t+1}-\nabla\L_\lambda^*(\w_{t+1})}$
and the variance
$\E_{t+1}\normsq{\hat G^w_{t+1}-G^w_{t+1}}$.
These terms are weighted by $\beta$ and $\beta^2$, respectively. Thus, the moving average suppresses the variance more strongly than the squared bias, making tight control of the bias particularly important. For the variance, we leverage the following bounds (see Lemma~\ref{lem:grad_w_var}):
\begin{align}
    &\text{Under Assumption~\ref{asm:oracle}:}\quad
    \E_{t+1}\normsq{\hat G^w_{t+1}-G^w_{t+1}}
    \leq \frac{1}{B}\left(2\sigma_f^2+8\lambda^2\sigma_g^2\right),
    \label{eq:upper_var_main}\\
    &\text{Under Assumptions~\ref{asm:oracle} and~\ref{asm:mss}:}\quad
    \E_{t+1}\normsq{\hat G^w_{t+1}-G^w_{t+1}}
    \leq \frac{1}{B}\left(2\sigma_f^2+\frac{8\tilde L_g^2C_f^2}{\mu_g^2}\right).
    \label{eq:upper_var_mss_method}
\end{align}
In particular, mean-square smoothness removes the dependence on $\lambda$ from the variance bound.

For the squared bias, our analysis relies on the following decomposition (see Lemma~\ref{lem:grad_decomp}):
\begin{align}
    \normsq{G_t^w-\nabla\L_\lambda^*(\w_t)}
    \leq O(1)Y_t
    +O(\lambda^2)S_t
    +\frac{O(\kappa^4)}{\lambda^2}.
    \label{eq:grad_est_decomp}
\end{align}
The arguments up to this point largely follow those developed in prior work~\citep{yang2026compositional,kwon2024complexity}.

It remains to establish recursions for the tracking errors $Y_t$ and $S_t$. The recursion for $Y_t$ follows standard SGD analysis for strongly convex objectives. The main departure from prior analyses lies in controlling $S_t$. Let
$G_t^s\coloneqq\E_t[\hat G_t^y-\hat G_t^z]$
and $\s_t^*\coloneqq\s^*(\w_t)$.
The analysis of~\citet{kwon2024complexity} relies on an inequality of the form
\begin{align}\label{eq:wrong_crossterm_bound}
    \inner{G_t^s}{\s_t-\s_t^*}\geq c\normsq{\s_t-\s_t^*}
\end{align}
for some $c>0$. However, this inequality does not hold in general, as demonstrated by the counterexample in Appendix~\ref{sec:counter_eg}.

To establish a valid contraction bound, we introduce the mapping
\begin{align}
    H(\w,\y,\s)
    \coloneqq \lambda^{-1}\nabla_2f(\w,\y)
    +\nabla_2g(\w,\y)
    -\nabla_2g(\w,\y-\s).
    \label{eq:H_method}
\end{align}
For fixed $(\w,\y)$, the mapping $H(\w,\y,\cdot)$ is $\mu_g$-strongly monotone. Using this property, we establish the following lower bound (see Lemma~\ref{lem:cross_gradnormsq_lb}):
\begin{align}
    \inner{G_t^s}{\s_t-\s_t^*}
    \geq \frac{\mu_g}{2}\normsq{\s_t-\s_t^*}
    -\frac{1}{\mu_g}
    \left(\frac{L_f^2}{\lambda^2}+Q_g^2r_\lambda^2\right)
    \normsq{\y_t-\y_\lambda^*(\w_t)}.
    \label{eq:s_cross_method}
\end{align}
This bound accounts explicitly for the $\y$-tracking error. Since $r_\lambda=C_f/(\mu_g\lambda)$, the coefficient of this additional error term is $O(\lambda^{-2})$. This bound is central to deriving the one-step recursion for the $\s$-tracking error in Lemma~\ref{lem:s_desc}.

% A second novel element is a sharper Lipschitz bound for $\s^*(\cdot)$ (see Lemma~\ref{lem:s_lip}):
% \begin{align}
%     \normsq{\s^*(\w_t)-\s^*(\w_{t-1})}
%     \leq \frac{O(\kappa^6)}{\lambda^2}
%     \normsq{\w_t-\w_{t-1}}.
%     \label{eq:s_lip_method}
% \end{align}
% This establishes that $\s^*(\cdot)$ is $O(\lambda^{-1})$-Lipschitz and eliminates the additive $O(\lambda^{-4})$ term in the corresponding bound of~\citet{kwon2024complexity}. Removing this residual term is crucial for obtaining the desired $O(\epsilon^{-6})$ complexity.

Combining these bounds yields Theorems~\ref{thm:std} and~\ref{thm:mss}. The complete proofs are provided in Appendices~\ref{sec:std} and~\ref{sec:mss}.

\section{Numerical experiments}\label{sec:exp}
We evaluate SICO on two stochastic bilevel learning tasks: data hyper-cleaning and feature-wise regularization selection. We compare against F2SA and F3SA~\citep{kwon2023fully}, PMLLC\footnote{The full name is \emph{Penalty Methods with Lower-Level Coupling}. Since the original authors did not provide an official abbreviation, we introduce one here for ease of presentation.}~\citep{kwon2024complexity}, F2BSA~\citep{JMLR:v26:23-1104}, and
Stoc-SGHA~\citep{gu2026sgha}. We also add a uniformly small weight baseline as advised by reviewer 9hen.

\noindent{\bf Implementation and hyperparameter tuning.}
All methods use logistic regression with an intercept, mini-batches of size $64$, and a budget of $10{,}000$ inner iterations, with evaluation every $100$ iterations. We run each experiments with 3 fixed seeds and report the average performance with standard deviation. Training and validation batches are sampled independently. Double loop method F2BSA and PMLLC perform five inner updates per outer update. We tune upper level learning rate $\eta$ in $\{1, 2, 3, 5\}*10^{\{-2,-1,0,1,2,3,4,5\}}$ and lower level learning rate $\gamma$ in $\{1, 2, 3, 5\}*10^{\{-5,-4,-3,-2,-1\}}$; $r_\lambda$ in $\{1, 2, 3, 5\}*10^{\{-3,-2,-1\}}$; $\lambda$ in $\{10, 20, 30, 100, 300, 1000, 10000\}$. For more tuning detail please refer to Appendix~\ref{sec:exp_detail}. Experiments are conducted on an NVIDIA A30 GPU.
\iffalse
We evaluate our method on three representative stochastic bilevel optimization problems: a synthetic nonconvex--strongly-convex problem, data hyper-cleaning, and regularization selection.
\subsection{Synthetic Problem(TBD)}

We first consider the bilevel problem
\begin{equation}
\min_{\w\in\mathbb{R}^{d_w}}
F(\w):=f\bigl(\w,\y^*(\w)\bigr), \qquad \y^*(\w):=\arg\min_{\y\in\mathbb{R}^{d_y}}g(\w,\y),
\end{equation}
where
\begin{align}
f(\w,\y)&=\frac{1}{r}\mathbf{1}^{\top}\sin(C\w+D\y)
+\frac{\rho}{2}|\w|^2,\
g(\w,\y)&=
\frac{1}{2}\y^\top A\y-\y^\top(B\w+c).
\end{align}
Here, $A$ is symmetric positive definite, so $g(\w,\cdot)$ is strongly convex and
\begin{equation}
\y^*(\w)=A^{-1}(B\w+c).
\end{equation}
The sine term makes the upper-level objective nonconvex. To introduce stochasticity in both levels, we use
\begin{align}
f_{\xi}(\w,\y)&=f(\w,\y)+u_{\xi}^{\top}\w+v_{\xi}^{\top}\y,\
g_{\zeta}(\w,\y)
&=g(\w,\y)+p_{\zeta}^{\top}\w+q_{\zeta}^{\top}\y,
\end{align}
where the noise vectors are independent and zero mean. Their variances are controlled by parameters $\sigma_f^2$ and $\sigma_g^2$. We generate $A$ with prescribed eigenvalues to examine different lower-level condition numbers and consider several stochastic noise levels.
\fi

\noindent{\bf Data Hyper-Cleaning.}
\label{subsec:hypercleaning}
In data hyper-cleaning, the training set contains corrupted labels, while a small clean validation set is available. An upper-level variable $\w_i$ is introduced to control the importance of each training example. Let
$
    \mathcal{D}_{\mathrm{tr}} = \{(a_i, b_i)\}_{i=1}^{n},  \mathcal{D}_{\mathrm{val}} =
    \{(a_j^{\mathrm{val}},b_j^{\mathrm{val}})\}_{j=1}^{m}.
$
We consider
\begin{align*}
&\min_{\w\in\mathbb{R}^{n}}\frac{1}{m}\sum_{j=1}^{m}\ell\left(\y^*(\w);
a_j^{\mathrm{val}},b_j^{\mathrm{val}}\right)\\
&\y^*(\w)=\arg\min_{\y}\frac{1}{n}\sum_{i=1}^{n}
\sigma(\w_i)\ell(\y;a_i, b_i)+\frac{\mu}{2}\normsq{\y},
\end{align*}
where $\sigma(\cdot)$ is the sigmoid function and $\ell$ is the cross-entropy loss. We use a multiclass logistic regression classifier, and the quadratic regularizer guarantees strong convexity of the lower-level problem. Stochastic upper- and lower-level gradients are computed using independently sampled validation and training mini-batches.

We use \textsc{MNIST} and \textsc{Fashion-MNIST} to verify our algorithm on this task. %Following the standard setting, a fraction of the training labels is randomly corrupted, while the validation and test sets remain clean.
MNIST and Fashion-MNIST ~\citep{lecun1998gradient,xiao2017fashion} contain $28\times28$ grayscale images of handwritten digits and clothing items, respectively, with ten classes each. For each dataset, we randomly partition the official $60{,}000$ training images into $50{,}000$ training and $10{,}000$ validation examples and retain the official $10{,}000$ test images. Each training label is independently replaced by a uniformly sampled incorrect class with probability $p\in\{0.5,0.9,0.99\}$; validation and test labels remain clean. We fix $\mu=0.1$. Different from previous studies~\cite{JMLR:v26:23-1104, gu2026sgha}~\footnote{We thank reviewer 9hen for pointing out that the mean garbage weight can be reduced by assigning uniformly small weights to all training samples. As shown in Appendix~\ref{subsec:extra_exp_fig}, this strategy degrades performance in our experiments. Uniform weighting treats clean and corrupted samples equally and therefore provides no selective data cleaning.}, we report the average weight separation on the partially corrupted training set which is defined as $\Delta=|G_+|^{-1}\sum_{i\in G^+}\sigma(\w_i)-|G_-|^{-1}\sum_{i\in G^-}\sigma(\w_i)$, where $G_+, G_-$ are the clean set and the corrupted set. Figure~\ref{fig:data_hypercleaning} shows that our algorithm learns a largest weight separation for and the convergence is the fastest. 

\begin{figure}[t]
\begin{center}
%\framebox[4.0in]{$\;$}
% \fbox{\rule[-.5cm]{0cm}{4cm} \rule[-.5cm]{4cm}{0cm}}
\includegraphics[scale=0.35]{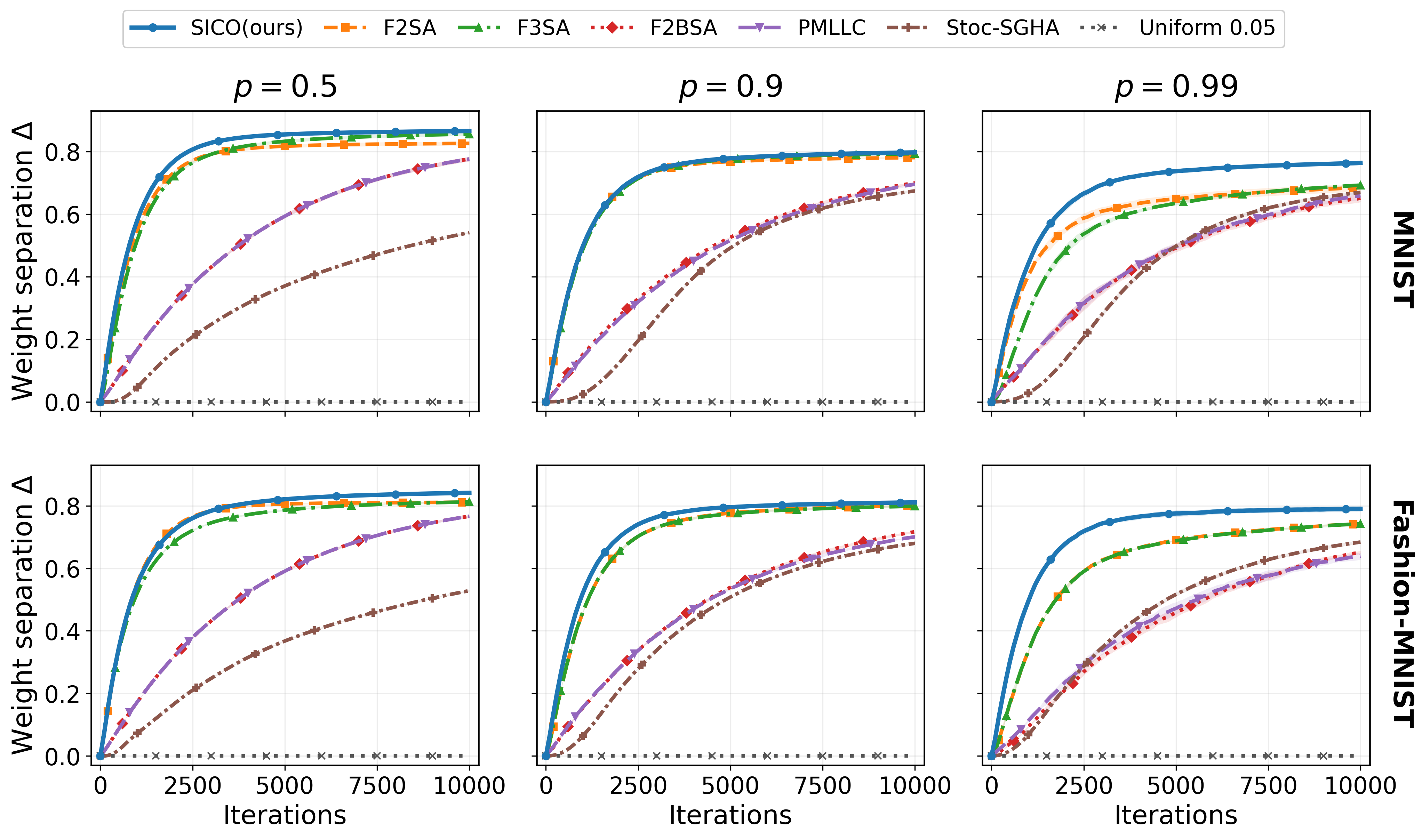}
\end{center}
% \vspace*{-0.2in}
\caption{Comparison of first order method of NC-SC SBO problem for data hyper-cleaning task. Upper row: MNIST dataset; Lower row: Fashion-MNIST dataset.}\label{fig:data_hypercleaning}
%\end{figure}
%\begin{figure}[t]
\begin{center}
\includegraphics[scale=0.25]{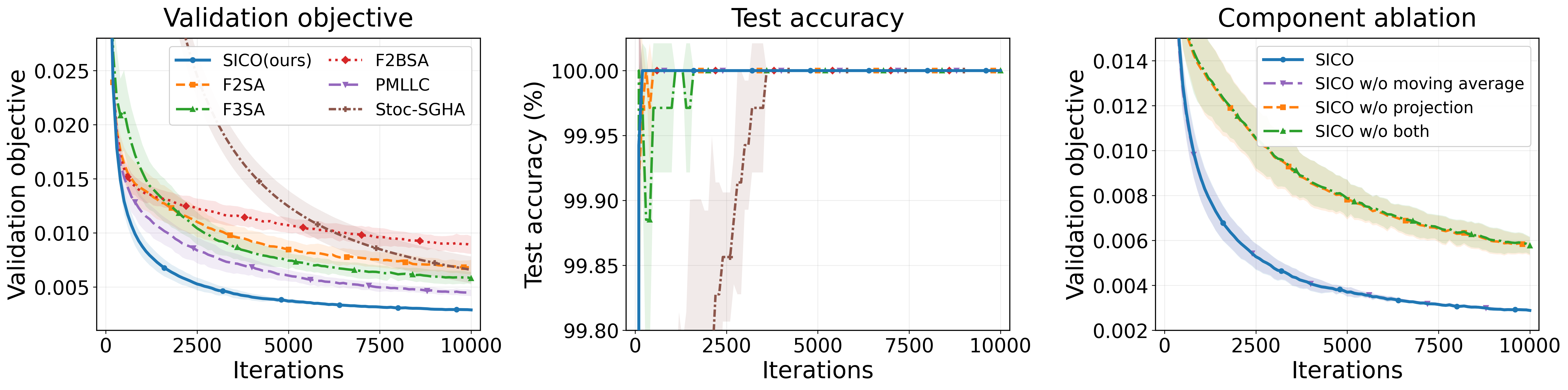}
\end{center}
% \vspace*{-0.2in}
\caption{Left two: Comparison of first order method of NC-SC SBO problem for regularization selection task on the Mushroom dataset. Right: Projection-by-momentum ablation study on the Mushroom dataset.}\label{fig:learn_regularizer}
\end{figure}

%\subsection
\noindent{\bf Regularization Selection.}
\label{subsec:regularization-selection}
We next consider learning feature-wise regularization parameters for logistic regression. Let $\y$ denote the classifier parameters and let $\w_k$ control the regularization strength of the $k$-th feature. We solve
% \vspace{-0.1pt}
\begin{align*}
&\min_{\w\in\mathbb{R}^{d}}\frac{1}{m}\sum_{j=1}^m\ell\left(\y^*(\w);a_j^{\text{val}},b_j^{\text{val}}\right)\\
&\y^*(\w)=\arg\min_{\y}\frac{1}{n}\sum_{i=1}^n
\ell(\y;a_i,b_i)+\frac{1}{2d}\sum_{k=1}^{d}\exp(w_k)\normsq{\y_{:,k}}+\frac{\mu_0}{2d}\normsq{\y}_F.
\end{align*}
% \vspace{-0.1pt}
% The fixed regularization parameter $\mu_0>0$ ensures that the lower-level objective is strongly convex.  Independent mini-batches from the training and validation sets provide stochastic gradients for the two levels.

% We use \textsc{Mushroom}, a binary-clasSICOcation datasets for verifying our algorithm.
We use the LIBSVM version of Mushroom dataset~\citep{CC01a}, which contains $8124$ examples from two classes and $112$ binary features encoding categorical attributes. A seeded random partition yields $5802/1161/1161$ training/validation/test examples. Features are standardized using training-set means and standard deviations, with unit scale for constant features; labels are mapped to $\{0,1\}$. We fix $\mu_0=0.1$ and divide both quadratic penalties by the number of feature weights ($d=112$). The intercept receives only the fixed penalty. Figure~\ref{fig:learn_regularizer} (left two) reports validation cross-entropy and test accuracy and shows that our method reach the lowest validation error.

Finally, we conduct an ablation study on the regularization selection task comparing SICO with the variant without the projection step on $\z$. The result shown in Figure~\ref{fig:learn_regularizer} shows that projection improves the convergence of SICO. 

%-------------------------
\iffalse
\begin{figure}[h]
\begin{center}
%\framebox[4.0in]{$\;$}
\fbox{\rule[-.5cm]{0cm}{4cm} \rule[-.5cm]{4cm}{0cm}}
\end{center}
\caption{Sample figure caption.}
\end{figure}

\begin{table}[t]
\caption{Sample table title}
\label{sample-table}
\begin{center}
\begin{tabular}{ll}
\multicolumn{1}{c}{\bf PART}  &\multicolumn{1}{c}{\bf DESCRIPTION}
\\ \hline \\
Dendrite         &Input terminal \\
Axon             &Output terminal \\
Soma             &Cell body (contains cell nucleus) \\
\end{tabular}
\end{center}
\end{table}
\fi
\section{Conclusions}
% In this work, we propose a fully first-order methods for NC-SC stochastic bilevel optimization. Building on the penalty framework and projected lower-level coupling, we develop a single-loop method that requires only a constant number of stochastic gradient samples per iteration. The key ingredients are to extend the projected tracking of the gap between the penalized and original lower-level solutions to the standard bounded-variance setting, and to combine it with a momentum estimator that controls the stochastic error of the penalty gradient without using a large minibatch. We further provide a refined analysis of the coupled lower-level dynamics which is necessary to provide the desired sample complexity. Under standard bounded variance assumption(resp. mean-square smoothness), our method achieves $O(\kappa^{11}\epsilon^{-6})$(resp. $O(\kappa^9\epsilon^{-4})$) complexity with $O(1)$ batch size, matching the best-known rate of existing first-order methods. These results show that neither a large batchsize nor a double-loop structure is necessary for attaining the best-known complexity within the first-order methods.

% Potential directions for future work is to further improve the dependence on the condition parameter $\kappa$ and closing the gap between this upper bound and the best known $O(\epsilon^{-4})$ lower bound~\citep{chen2025condition}.

We have proposed a fully first-order method for non-convex-strongly-convex stochastic bilevel optimization. Building on the penalty framework and projected lower-level coupling, we combine projected tracking of lower level variables with a momentum update of the upper level variable to obtain a single-loop algorithm using only $O(1)$ batch size. Our refined analysis establishes $O(\kappa^{11}\epsilon^{-6})$ complexity under standard bounded variance and $O(\kappa^9\epsilon^{-4})$ complexity under the additional mean-square smoothness assumption, matching the best-known rates of existing fully first-order methods without requiring large batches or nested loops. %Future directions include improving the dependence on the condition number $\kappa$ and closing the gap between the current upper bound and the best-known $O(\epsilon^{-4})$ lower bound~\citep{chen2025condition}.

\newpage
% \section{Final instructions}
% Do not change any aspects of the formatting parameters in the style files.
% In particular, do not modify the width or length of the rectangle the text
% should fit into, and do not change font sizes (except perhaps in the
% \textsc{References} section; see below). Please note that pages should be
% numbered.

\section{Statements}
\subsection*{AI use statement}
In this work, we used generative AI tools to polish the main paper writing. We also use them to help build the codebase for numerical experiments presented in section~\ref{sec:exp}. We have reviewed all the AI-assisted work. Specifically, 2 of the authors reviewed and verified the training and evaluation pipeline and algorithms implementation(including our method and baseline methods) to make sure they are correct. Besides, we use AI tools to help construct the counter example in appendix~\ref{sec:counter_eg}. All the remaining required disclosure tasks are either inapplicable to this work or are done by ourselves. We take responsibility for the final content of this work, including text, claims or artifacts produced with the aid of generative AI.

\subsection*{Reproducibility statement}
For theoretical results, we provide formal assumption and theorem and necessary explanation in the main paper, and technical lemma with their proof in appendix~\ref{sec:lem}, and proof of theorems in appendix~\ref{sec:std} and~\ref{sec:mss}. 

For experiment results, we manually select 3 random seeds for each experiment and provide the detail information about experiment setting in section~\ref{sec:exp}. All the datasets used in our experiment are open source. All the aforementioned efforts help make sure that all of presented result are reproducible.

\iffalse
(This section is \textbf{recommended} and does not count toward the page limit.)

It is important that the work published in ICLR is reproducible. Authors are
strongly encouraged to include a paragraph-long Reproducibility Statement at the
end of the main text (before references) to discuss the efforts that have been
made to ensure reproducibility. This paragraph should not itself describe
details needed for reproducing the results, but rather reference the parts of
the main paper, appendix, and supplemental materials that will help with
reproducibility. For example, for novel models or algorithms, a link to an
anonymous downloadable source code can be submitted as supplementary materials;
for theoretical results, clear explanations of any assumptions and a complete
proof of the claims can be included in the appendix; for any datasets used in
the experiments, a complete description of the data processing steps can be
provided in the supplementary materials. Each of the above are examples of
things that can be referenced in the reproducibility statement.
\fi
\iffalse
\subsubsection*{Author Contributions}
If you'd like to, you may include  a section for author contributions as is done
in many journals. This is optional and at the discretion of the authors.

\subsubsection*{Acknowledgments}
Use unnumbered third level headings for the acknowledgments. All
acknowledgments, including those to funding agencies, go at the end of the paper.
\fi
\newpage
\bibliography{iclr2027_conference,ref}
\bibliographystyle{iclr2027_conference}

\clearpage
\appendix

\section{Technical Lemma}\label{sec:lem}

\subsection{Preliminary}
\paragraph{Notation}
Recall that 
\begin{align*}
    \s_t &\coloneqq \y_t-\z_t,\; \s^*(\w_t) \coloneqq \y_{\lambda}^*(\w_t) - \y^*(\w_t),\\  Y_t \coloneqq \normsq{\y_{t+1}-\y_\lambda^*(\w_t)},\; &S_t \coloneqq \normsq{\s_{t+1}-\s^*(\w_t)},\; V_t \coloneqq \normsq{\v_{t+1}-\nabla\L_\lambda^*(\w_t)}.
\end{align*}
For convenience, we further introduce the following notations:
\[\y_{\lambda, t}^*\coloneqq\y_\lambda^*(\w_t), \s_t^*\coloneqq\s^*(\w_t),\]
and denote $\I$ as identity matrix. Besides, let us denote expectation given the history before the $t$-th iteration(resp. given the history before the $t$-th iteration together with $\B^t_{fy}$ and $\B^t_{gy}$) as $\E_t$(resp. $\E_{t+}$). 

Below we will introduce several useful lemma. 
\begin{lem}\label{lem:ll_sol_prox}
    Under Assumption~\ref{asm:func}, for $\lambda\geq\frac{2L_f}{\mu_g}$, it holds that
    \begin{align}\label{eq:ll_sol_prox}
        \Norm{\y^*(\w) - \y_\lambda^*(\w)}\leq \frac{C_f}{\mu_g\lambda}
    \end{align}
\end{lem}
\begin{proof}
    It holds from definition of $\y_\lambda^*(\w)$ that $\nabla_2 f(\w, \y_\lambda^*(\w)) + \lambda \nabla_2 g(\w, \y_\lambda^*(\w)) = \textbf{0}$. Then we have
    \begin{align*}
        \Norm{\y^*(\w) - \y_\lambda^*(\w)}
        \leq \frac{1}{\mu_g}\Norm{\nabla_2 g(\w, \y_\lambda^*(\w))}
        = \frac{1}{\mu_g\lambda}\Norm{\nabla_2 f(\w, \y_\lambda^*(\w))}
        \leq \frac{C_f}{\mu_g\lambda}  
    \end{align*}
    where the first inequality holds from strongly convexity of $g(\w, \cdot)$.
\end{proof}
\begin{lem}\label{lem:y_lip}
    Under Assumption~\ref{asm:func}, for $\lambda\geq \frac{2L_f}{\mu_g}$, $\y_\lambda^*(\cdot)$ is $C_{y*} = \frac{2(L_f+\lambda L_g)}{\lambda \mu_g} = O(\kappa)$-lipschitz.
\end{lem}
\begin{proof}
It holds from the definition of $\y_\lambda^*(\w)$ that $\nabla_2 \L_\lambda(\w, \y_\lambda^*(\w)) = \textbf{0}$. Differentiate LHS w.r.t. $\w$, we have 
\begin{align*}
    \nabla_{21}^2 \L_\lambda(\mathbf{w}, y^*_\lambda(\mathbf{w})) + \nabla y^*_\lambda(\mathbf{w})\nabla_{22}^2 \L_\lambda(\mathbf{w}, y^*_\lambda(\mathbf{w}))= \textbf{0}.
\end{align*}
Rearranging the term and multiplying both side by $\nabla_{22}^2 \L_\lambda^*(\mathbf{w}, y^*_\lambda(\mathbf{w}))^{-1}$ yields
\begin{align}\label{eq:grad_y_lamstar}
    \nabla y^*_\lambda(\w) = -\left(\frac{1}{\lambda}\nabla^2_{21}f(\w, y^*_\lambda(\w)) + \nabla^2_{21}g(\w, y^*_\lambda(\w))\right)\left(\frac{1}{\lambda}\nabla^2_{22}f(\w, y^*_\lambda(\w)) + \nabla^2_{22}g(\w, y^*_\lambda(\w))\right)^{-1}
\end{align}
It finishes the proof by noting that $\nabla^2_{22}f(\w, y^*_\lambda(\w))\succeq -L_f\I, \nabla^2_{21}f(\w, y^*_\lambda(\w))\preceq L_f\I$ and $\nabla^2_{22}g(\w, y^*_\lambda(\w))\succeq \mu_g\I, \nabla^2_{21}g(\w, y^*_\lambda(\w))\preceq L_g\I$
\end{proof}
\begin{lem}[Lemma B.5 of~\citet{JMLR:v26:23-1104}]\label{lem:s_lip}
Under Assumption~\ref{asm:func}, if $\lambda\geq \frac{2L_f}{\mu_g}$, we have that for any $t$:
\begin{align}\label{eq:s_lip}
    \normsq{\s^*_t - \s^*_{t-1}} = \normsq{\s^*(\w_t) - \s^*(\w_{t-1})}\leq \frac{C_s^2}{\lambda^2}\normsq{\w_t - \w_{t-1}}
\end{align}
% where $C_s = \frac{2(1+\kappa)(L_f+\frac{Q_gC_f}{\mu_g})}{\mu_g} = O(\kappa^3)$. 
where $C_s = (\frac{1}{\mu_g}+\frac{2L_g}{\mu_g^2})(L_f + \frac{C_fQ_g}{\mu_g}) = O(\kappa^3)$. 
\end{lem}

\subsection{Lemma for Theorem~\ref{thm:std}}
In this section we introduce several important technical lemma for Theorem~\ref{thm:std}.
%We start from a non-convex starter[proposition A.3 in~\cite{kwon2024complexity}]: 

The following lemma is a standard descent lemma for smooth function $\L_\lambda^*$.
\begin{lem}\label{lem:descent}
Under Assumption~\ref{asm:func},~\ref{asm:oracle}, for $\eta\leq \frac{1}{2L_*}$ it holds that
\begin{align}\label{eq:noncvx_starter}
    \E \L_\lambda^*(\w_{t+1}) \leq \E\L_\lambda^*(\w_t) + \frac{\eta}{2}\E\normsq{\v_{t+1} - \nabla\L_\lambda^*(\w_t)} - \frac{\eta}{4}\E\normsq{\v_{t+1}} - \frac{\eta}{2}\E\normsq{\nabla\L_\lambda^*(\w_t)}.
\end{align}    
\end{lem}
\begin{proof}%[proof of Lemma~\ref{lem:descent}]
    The proof is quite standard. We present for completeness. 
\begin{align}
    \E_{t+} \L_\lambda^*(\w_{t+1}) &\leq \L_\lambda^*(\w_t) + \E_{t+}\langle\nabla \L_\lambda^*(\w_t), \w_{t+1} - \w_t\rangle + \frac{L_*}{2}\E_{t+}\normsq{\w_{t+1} - \w_t}\nonumber\\
    &\leq \L_\lambda^*(\w_t) - \eta\E_{t+}\langle\nabla \L_\lambda^*(\w_t), \v_{t+1}\rangle + \frac{\eta^2 L_*}{2}\E_{t+}\normsq{\v_{t+1}}\nonumber\\
    &\leq \L_\lambda^*(\w_t) + \frac{\eta}{2}\left(\E_{t+}\normsq{\v_{t+1} - \nabla\L^*_\lambda(\w_t)} - \E_{t+}\normsq{\v_{t+1}} - \normsq{\nabla\L^*_\lambda(\w_t)}\right) \nonumber\\&\qquad+ \frac{\eta^2 L_*}{2}\E_{t+}\normsq{\v_{t+1}}\nonumber\\
    &\leq \L_\lambda^*(\w_t) + \frac{\eta}{2}\E_{t+}\normsq{\v_{t+1} - \nabla\L^*_\lambda(\w_t)} - \frac{\eta}{4}\E_{t+}\normsq{\v_{t+1}} - \frac{\eta}{2}\normsq{\nabla\L^*_\lambda(\w_t)}    
\end{align}
where the first inequality holds from the $L_*$-smoothness of $\L_\lambda^*(\cdot)$ and the last inequality holds from $\eta\leq \frac{1}{2L_*}$. Finally, taking expectation w.r.t. all the randomness before $t$-th iteration, $\B^t_{fy}$ and $\B^t_{gy}$ finishes the proof by the law of total expectation.
\end{proof}
We then provide a recursive bound for the error of the momentum estimator.
\begin{lem}[Lemma A.2 of~\citet{kwon2024complexity}]\label{lem:grad_w_var}
Under Assumption~\ref{asm:func},~\ref{asm:oracle}, it holds that:
\begin{align}\label{eq:w_var}
    \text{var}(\hat{G_{t+1}^w}) \coloneqq \E_{(t+1)+}\normsq{\hat{G}_{t+1}^w - G_{t+1}^w} \leq \frac{1}{B}\left(2\sigma_f^2 + 8\lambda^2\sigma_g^2\right).
\end{align}
If we further assume that Assumption~\ref{asm:mss} holds and let $r_\lambda = \frac{C_f}{\mu_g \lambda}$, we have that
\begin{align}\label{eq:w_var_mss}
    \text{var}(\hat{G_{t+1}^w}) \coloneqq \E_{(t+1)+}\normsq{\hat{G}_{t+1}^w - G_{t+1}^w} \leq \frac{1}{B}\left(2\sigma_f^2 + 8\frac{\tilde{L}_g^2C_f^2}{\mu_g^2}\right).
\end{align}
\end{lem}
% We then leverage the bound for $\normsq{G^w_t - \nabla\L_\lambda^*(\w_t)}$ in lemma A.1(ii)~\cite{kwon2024complexity} by setting projection radius to be $r_\lambda = \frac{C_f}{\mu_g\lambda}$:
The following lemma decomposes the error of the (expected) hypergradient estimator.
\begin{lem}[lemma A.1(ii) of~\cite{kwon2024complexity}]\label{lem:grad_decomp}
Under Assumption~\ref{asm:func}, let $r_\lambda = \frac{C_f}{\mu_g\lambda}$, it holds that
\begin{align}\label{eq:projed_grad_est_unpack}
    \normsq{G^w_t - \nabla\L_\lambda^*(\w_t)}\leq 3L_{y*}^2\normsq{\y_{t+1} - \y^*_\lambda(\w_t)} + 3\lambda^2 L_g^2\normsq{\s_{t+1} - \s^*(\w_t)} + \frac{3C_f^2L_{y*}^2}{\mu_g^2\lambda^2}
\end{align}
where $L_{y*} = L_f + \frac{C_fQ_g}{\mu_g} = O(\bar{\ell}\kappa)$.    
\end{lem}
\textbf{Remark:} The original lemma in~\citet{kwon2024complexity} requires (MSS) Assumption~\ref{asm:mss}. However, note that LHS doesn't involve any randomness/stochastic gradient(when $\w_t$ is fixed), we point out that this assumption is not needed. In fact, $\tilde L_g$ doesn't appear in the RHS.
% From lemma 3.2 in~\cite{kwon2023fully} we have that $\y^*_\lambda(\w_t)$ is $C_{y*} = \frac{2(L_f+\lambda L_g)}{\lambda \mu_g} = O(\kappa)$-lipschitz, we can handle the contraction of the 1st term in RHS in a classic way while the most difficult part is the second term.
\begin{lem}\label{lem:mmt_recur}    
Under Assumption~\ref{asm:func},~\ref{asm:oracle}, for $\beta\leq 1$ it holds that
\begin{align}\label{eq:mmt_err_bi_penal}
    \E V_{t+1} \leq (1-\beta)\E V_t + \frac{2L_*^2\eta^2}{\beta}\E\normsq{\v_{t+1}} + 4\beta\E\normsq{G^w_{t+1} - \nabla\L_\lambda^*(\w_{t+1})} + \frac{\beta^2\lambda^2\sigma^2}{B}
\end{align}
where $G_{t+1}^w = \E_{(t+1)+} \hat{G}_{t+1}^w$ and $\lambda^2\sigma^2 = 2\sigma_f^2 + 8\lambda^2\sigma_g^2$. 
\end{lem}
\begin{proof}%[proof of Lemma~\ref{lem:mmt_recur}]
The proof is quite standard and we present for completeness.
\begin{align}
    \E_{(t+1)+} V_{t+1} &= \E_{(t+1)+} \normsq{\beta \hat{G}_{t+1}^w + (1-\beta)\v_{t+1} - \nabla\L_\lambda^*(\w_{t+1})}\nonumber\\
    &= \E_{(t+1)+} \|\beta(\hat{G}_{t+1}^w - G_{t+1}^w) + \beta(G_{t+1}^w - \nabla\L_\lambda^*(\w_{t+1}))\nonumber\\ &\qquad+ (1-\beta)(\v_{t+1} - \nabla\L_\lambda^*(\w_t) + \nabla\L_\lambda^*(\w_t) - \nabla\L_\lambda^*(\w_{t+1}))\|^2\nonumber\\
    &\leq \frac{\beta^2\lambda^2\sigma^2}{B} + (1-\beta)^2(1+\beta)\normsq{\v_{t+1} - \nabla\L_\lambda^*(\w_t)}\nonumber\\ &\qquad+ (1+\frac{1}{\beta})\left(2\beta^2\normsq{G_{t+1}^w - \nabla\L_\lambda^*(\w_{t+1})} + 2(1-\beta)^2\normsq{\nabla\L_\lambda^*(\w_t) - \nabla\L_\lambda^*(\w_{t+1})}\right)\nonumber\\
    &\leq \frac{\beta^2\lambda^2\sigma^2}{B} + (1-\beta)\normsq{\v_{t+1} - \nabla\L_\lambda^*(\w_t)}\nonumber\\ &\qquad+ 4\beta\normsq{G_{t+1}^w - \nabla\L_\lambda^*(\w_{t+1})} + \frac{2}{\beta}\normsq{\nabla\L_\lambda^*(\w_t) - \nabla\L_\lambda^*(\w_{t+1})}\nonumber\\
    % &\leq \beta^2\lambda^2\sigma^2 + \beta^2(1+\frac{1}{\beta})\E_t\normsq{G_{t+1}^w - \nabla\L_\lambda^*(\w_{t+1})}\nonumber\\ &\qquad+ (1-\beta)\left((1+\frac{\beta}{2})\normsq{\v_{t+1} - \nabla\L_\lambda^*(\w_t)} + (1+\frac{2}{\beta})\normsq{\nabla\L_\lambda^*(\w_t) - \nabla\L_\lambda^*(\w_{t+1})}\right)\nonumber\\
    &\leq \frac{\beta^2\lambda^2\sigma^2}{B} + (1-\beta)\normsq{\v_{t+1} - \nabla\L_\lambda^*(\w_t)} + 4\beta\normsq{G_{t+1}^w - \nabla\L_\lambda^*(\w_{t+1})} + \frac{2L_*^2\eta^2}{\beta}\normsq{\v_{t+1}}
\end{align}
where the first inequality holds from~\ref{eq:w_var} and young's inequality, the second inequality holds from $\beta\leq 1$. Finally, taking expectation w.r.t. all the randomness before $t+1$-th iteration, $\B^{t+1}_{fy}$ and $\B^{t+1}_{gy}$ finishes the proof by the law of total expectation.
\end{proof}

\begin{lem}[Descent lemma for $\y$]\label{lem:y_recur}    
Under Assumption~\ref{asm:func},~\ref{asm:oracle}, for $\gamma \leq \frac{\mu_y}{L_y^2}$, we have 
\begin{align}\label{eq:y_desc}
    \E Y_t \leq (1-\frac{\gamma\mu_y}{2}) \E Y_{t-1} + \frac{2C_{y*}^2\eta^2}{\gamma\mu_y}\E \normsq{\v_{t}} + \frac{\gamma^2\sigma_y^2}{B}
\end{align}
where $\sigma_y^2 = \frac{2\sigma_f^2}{\lambda^2} + 2\sigma_g^2 = O(1)$ and $C_{y*} = \frac{2(L_f+\lambda L_g)}{\lambda \mu_g} = O(\kappa)$.
\end{lem}
\begin{proof}%[proof of Lemma~\ref{lem:y_recur}]
\begin{align}
    \E_t Y_t &= \E_t \normsq{\y_t - \gamma \hat{G}_t^y - \y_{\lambda, t}^*}\nonumber\\
    &= \normsq{\y_t - \y_{\lambda, t}^*} - 2\gamma\E_t\langle\hat{G}_t^y, \y_t - \y_{\lambda, t}^*\rangle + \gamma^2\E_t\normsq{\hat{G}_t^y} \nonumber\\
    &\leq \normsq{\y_t - \y_{\lambda, t}^*} - 2\gamma\mu_y\normsq{\y_t - \y_{\lambda, t}^*} + \gamma^2\left(L_y^2\normsq{\y_t - \y_{\lambda, t}^*} + \frac{\sigma_y^2}{B}\right) \nonumber\\
    &\leq (1-\gamma\mu_y) \normsq{\y_t - \y_{\lambda, t}^*} + \frac{\gamma^2\sigma_y^2}{B} \nonumber\\
    &\leq (1-\gamma\mu_y)(1+\frac{\gamma\mu_y}{2}) \normsq{\y_t - \y_{\lambda, t-1}^*} + (1-\gamma\mu_y)(1+\frac{2}{\gamma\mu_y})\normsq{\y_{\lambda, t-1}^* - \y_{\lambda, t}^*} + \frac{\gamma^2\sigma_y^2}{B} \nonumber\\
    &\leq (1-\frac{\gamma\mu_y}{2})\normsq{\y_t - \y_{\lambda, t-1}^*} + \frac{2C_{y*}^2\eta^2}{\gamma\mu_y}\normsq{\v_{t}} + \frac{\gamma^2\sigma_y^2}{B}  
\end{align}
where $\mu_y = \mu_g - \frac{L_f}{\lambda}$ and $L_y = \frac{L_f}{\lambda} + L_g$ are the strongly convexity coefficient and the smoothness coefficient of $\L_y^\lambda(\w_t, \cdot) = \frac{1}{\lambda}f(\w_t, \cdot) + g(\w_t, \cdot)$. The last equality holds from $G_t^{y*} = \nabla_2 \L_y^\lambda(\w_t, \y_{\lambda, t}^*) = \textbf{0}$, Lemma~\ref{lem:y_lip} and the inequality that $(1-\gamma \mu_y)(2+\gamma\mu_y)\leq 2$ when $0\leq\gamma\leq \frac{\mu_y}{L_y^2}$. The first inequality holds from the strongly convexity and smoothness of $\L_y^\lambda(\w_t, \cdot)$ and the second inequality holds from $\gamma \leq \frac{\mu_y}{L_y^2}$. Finally, taking expectation w.r.t. all the randomness before $t$-th iteration, $\B^t_{fy}$ and $\B^t_{gy}$ finishes the proof by the law of total expectation.
\end{proof}

% \paragraph{Aug 24, 2026 Update}
% We now provide a new version of proof that can patch the issue. Instead of directly leveraging the strongly convexity and get the problematic lower bound for cross term, we make a finer analysis for it. Below we first present a new corrected lower bound for cross term that need to be handled when developing the descent lemma for $\s$.
\begin{lem}\label{lem:cross_gradnormsq_lb}
Under Assumption~\ref{asm:func}, we have that the following two inequality hold:
\begin{align}
    \langle G_t^s, \s_t - \s_t^*\rangle &\geq \frac{\mu_g}{2}\normsq{\s_t - \s_t^*} - \frac{1}{\mu_g} \left(\frac{1}{\lambda^2} L_f^2 + Q_g^2r_\lambda^2\right)\normsq{\y_t - \y_\lambda^*(\w_t)}\label{eq:s_cross_lb},\\
    \normsq{G_t^s} &\leq 2L_g^2\normsq{\s_t - \s_t^*} + 2\left(\frac{2}{\lambda^2} L_f^2 + 2Q_g^2r_\lambda^2\right)\normsq{\y_t - \y_\lambda^*(\w_t)}.\label{eq:s_gradnorm_lb}   
\end{align}    
\end{lem}
\begin{proof}%[proof of Lemma~\ref{lem:cross_gradnormsq_lb}]    
We first show that the mapping $H(\w,\y, \s) \coloneqq \frac{1}{\lambda}\nabla_2 f(\w,\y) + \nabla_2 g(\w, \y) - \nabla_2 g(\w, \y-\s)$ is $\mu_g$-strongly monotone and $L_g$ lipschitz in $\s$, i.e. 
\begin{align}
    \langle H(\w, \y, \s_1) - H(\w, \y, \s_2), \s_1 - \s_2\rangle &\geq \mu_g \normsq{\s_1 - \s_2} ,\label{eq:H_smnt_cross}\\
    \mu_g^2\normsq{\s_1 - \s_2} \leq \normsq{H(\w, \y, \s_1) - H(\w, \y, \s_2)} &\leq L_g^2\normsq{\s_1 - \s_2}. \label{eq:H_smnt_sq}
\end{align}
The proof is straightforward, we directly plug the definition of $H$ in the LHS of~(\ref{eq:H_smnt_cross}):
\begin{align}
    \text{LHS} &= \langle \nabla_2 g(\w, \y-\s_2) - \nabla_2 g(\w, \y-\s_1), \s_1 - \s_2\rangle \nonumber\\
    &= \langle \nabla_2 g(\w, \y-\s_2) - \nabla_2 g(\w, \y-\s_1), (\y - \s_2) - (\y - \s_1)\rangle \nonumber\\
    &\geq \mu_g\normsq{\s_1 - \s_2}
\end{align}
where the inequality holds from the strongly convexity of $g$. Then the first inequality of~(\ref{eq:H_smnt_sq}) holds obviously by combining the fact that 
\[\Norm{H(\w, \y, \s_1) - H(\w, \y, \s_2)}\Norm{\s_1 - \s_2}\geq \langle H(\w, \y, \s_1) - H(\w, \y, \s_2), \s_1 - \s_2\rangle\]
with~(\ref{eq:H_smnt_cross}) while the second inequality holds from smoothness of $g$. \\
We note that $G_t^s = H(\w_t, \y_t, \s_t)$ and that 
\begin{align*}
    H(\w_t, \y_\lambda^*(\w_t), \s^*_t) &= H(\w_t, \y_\lambda^*(\w_t), \y_\lambda^*(\w_t) - \y^*(\w_t)) \\
    &= \frac{1}{\lambda}\nabla_2 f(\w_t,\y_\lambda^*(\w_t)) + \nabla_2 g(\w_t, \y_\lambda^*(\w_t)) - \nabla_2 g(\w_t, \y^*(\w_t))\\
    &= \textbf{0}-\textbf{0} = \textbf{0},
\end{align*}
then the cross term becomes:
\begin{align}
    \langle G_t^s, \s_t - \s_t^*\rangle &= \langle H(\w_t, \y_t, \s_t) - H(\w_t, \y_t, \s^*_t), \s_t - \s_t^*\rangle\nonumber\\&\qquad + \langle H(\w_t, \y_t, \s^*_t) - H(\w_t, \y_\lambda^*(\w_t), \s^*_t), \s_t - \s_t^*\rangle\nonumber\\
    &\geq \mu_g\normsq{\s_t - \s_t^*}+ \langle H(\w_t, \y_t, \s^*_t) - H(\w_t, \y_\lambda^*(\w_t), \s^*_t), \s_t - \s_t^*\rangle\nonumber\\
    &\geq \mu_g\normsq{\s_t - \s_t^*} - \frac{1}{2\mu_g} \normsq{H(\w_t, \y_t, \s^*_t) - H(\w_t, \y_\lambda^*(\w_t), \s^*_t)} - \frac{\mu_g}{2} \normsq{\s_t - \s_t^*}\nonumber\\
    &\geq \frac{\mu_g}{2}\normsq{\s_t - \s_t^*} - \frac{1}{2\mu_g} \normsq{H(\w_t, \y_t, \s^*_t) - H(\w_t, \y_\lambda^*(\w_t), \s^*_t)} .\label{eq:cross_split}
\end{align}
For the second term in RHS, we note that $H(\w, \y, \s) =  \frac{1}{\lambda}\nabla_2 f(\w,\y) + \int_0^1\nabla^2_{22} g(\w, \y-\tau\s)\s d\tau$, then
\begin{align}
    \normsq{H(\w_t, \y_t, \s^*_t) - H(\w_t, \y_\lambda^*(\w_t), \s^*_t)} &\leq \frac{2}{\lambda^2} L_f^2\normsq{\y_t - \y_\lambda^*(\w_t)} \nonumber\\ + 2\normsq{\s^*}&\normsq{\int_0^1\left(\nabla^2_{22} g(\w, \y_t-\tau\s) - \nabla^2_{22} g(\w, \y_\lambda^*(\w_t)-\tau\s)\right) d\tau} \nonumber\\
    &\leq \frac{2}{\lambda^2} L_f^2\normsq{\y_t - \y_\lambda^*(\w_t)} + 2Q_g^2\normsq{\s}\normsq{\y_t - \y_\lambda^*(\w_t)} \nonumber\\
    &\leq \left(\frac{2}{\lambda^2} L_f^2 + 2Q_g^2r_\lambda^2\right)\normsq{\y_t - \y_\lambda^*(\w_t)}.\label{eq:H_y_lips}
\end{align}
Plugging~(\ref{eq:H_y_lips}) into~(\ref{eq:cross_split}) leads to~(\ref{eq:s_cross_lb}). The~(\ref{eq:s_gradnorm_lb}) can be proved in a similar way: 
\begin{align*}
    \normsq{G_t^s} &\leq 2\normsq{H(\w_t, \y_t, \s_t) - H(\w_t, \y_t, \s_t^*)} + 2\normsq{H(\w_t, \y_t, \s_t^*)}\nonumber\\
    &\leq 2L_g^2\normsq{\s_t - \s_t^*} + 2\left(\frac{2}{\lambda^2} L_f^2 + 2Q_g^2r_\lambda^2\right)\normsq{\y_t - \y_\lambda^*(\w_t)}.
\end{align*}
The second inequality holds from~(\ref{eq:H_smnt_sq}) and~(\ref{eq:H_y_lips}).
\end{proof}

\begin{prop}[Proposition 3.4 from~\cite{kwon2024complexity}]\label{prop:s_proj_desc}
Denote $\bar s_t = \y_{t+1} - \bar \z_t$. Under Assumption~\ref{asm:func}, for any $t$, with $r_\lambda = \frac{C_f}{\lambda\mu_g}$, 
    \begin{align}
        \Norm{\s_{t+1} - \s_t^*} \leq \Norm{\bar{\s}_t - \s_t^*}.
    \end{align}
\end{prop}
\begin{lem}[Descent lemma for $\s$]\label{lem:s_desc}   
Under Assumption~\ref{asm:func},~\ref{asm:oracle}, let $\gamma \leq \min\{\frac{\mu_g}{4L_g^2}, \frac{1}{2\mu_g}\}$, then it holds that
\begin{align}\label{eq:s_desc_pre}
    \E S_t &\leq \E\normsq{\bar{\s}_t - \s_t^*}\nonumber\\
    &\leq \left(1-\frac{\gamma\mu_g}{4}\right)\E S_{t-1} + \frac{4}{\gamma\mu_g}\E \normsq{\s^*_t - \s^*_{t-1}} + \frac{8\gamma}{\mu_g}\left(\frac{L_f^2}{\lambda^2}+Q_g^2r_\lambda^2\right)\left(\E Y_{t-1} + \eta^2C_{y*}^2\E\normsq{\v_t} \right) \nonumber\\&\qquad+ \frac{2\gamma^2}{B}(\sigma_y^2 + \sigma_g^2).
\end{align}
\end{lem}
\begin{proof}%[proof of Lemma~\ref{lem:s_desc}]
The first inequality holds from proposition~\ref{prop:s_proj_desc}. We now prove the second inequality. 
\begin{align}\label{eq:s_pre_derive}
    \E_t\normsq{\bar{\s}_t - \s_t^*} &= \E_t\normsq{\s_t - \gamma \hat{G}_t^s - \s_t^*}\nonumber\\
    &= \E_t\normsq{\s_t - \s_t^*} + \gamma^2(\text{var}(\hat{G}_t^s) + \E_t\normsq{G_t^s}) - 2\gamma\langle G_t^s, \s_t - \s_t^*\rangle\nonumber\\
    &\leq \E_t\normsq{\s_t - \s_t^*} + \gamma^2\left(\frac{2\sigma_y^2 + 2\sigma_g^2}{B} + 2L_g^2\normsq{\s_t - \s_t^*}\right)\nonumber\\
    &\qquad + {\color{blue} 2\gamma^2\left(\frac{2}{\lambda^2} L_f^2 + 2Q_g^2r_\lambda^2\right)\normsq{\y_t - \y_\lambda^*(\w_t)}}\nonumber\\
    &\qquad  - \gamma\mu_g\normsq{\s_t - \s_t^*}{\color{blue}+ \frac{2\gamma}{\mu_g} \left(\frac{1}{\lambda^2} L_f^2 + Q_g^2r_\lambda^2\right)\normsq{\y_t - \y_\lambda^*(\w_t)}}\nonumber\\
    &\leq (1-\frac{\gamma\mu_g}{2})\E_t\normsq{\s_t - \s_t^*} + \frac{2\gamma^2}{B}(\sigma_y^2 + \sigma_g^2)\nonumber\\
    & {\color{blue}+ \frac{4\gamma(1+2\gamma\mu_g)}{\mu_g} \left(\frac{1}{\lambda^2} L_f^2 + Q_g^2r_\lambda^2\right)\left(\normsq{\y_t - \y_\lambda^*(\w_{t-1})} + \normsq{\y_\lambda^*(\w_{t-1}) - \y_\lambda^*(\w_t)} \right)}\nonumber\\
    &\leq (1-\frac{\gamma\mu_g}{4})\E_t\normsq{\s_t - \s_{t-1}^*} + \frac{4}{\gamma\mu_g}\normsq{\s_t^* - \s_{t-1}^*} + \frac{2\gamma^2}{B}(\sigma_y^2 + \sigma_g^2)\nonumber\\
    &\qquad {\color{blue}+ \frac{8\gamma}{\mu_g} \left(\frac{1}{\lambda^2} L_f^2 + Q_g^2r_\lambda^2\right)\left(Y_{t-1} + \eta^2C_{y*}^2\normsq{\v_t} \right)}
\end{align}
where $\hat{G}_t^s = \hat{G}_t^y - \hat{G}_t^z$, the last inequality holds from Young's inequality. And the blue part is the newly introduced term according to the corrected analysis of cross term presented in Lemma~\ref{lem:cross_gradnormsq_lb}. Finally, taking expectation w.r.t. all the randomness before $t$-th iteration, $\B^t_{fy}$ and $\B^t_{gy}$ finishes the proof by the law of total expectation.
\end{proof}

% From lemma A.7 in~\cite{kwon2024complexity}, we have that 
% \[\normsq{\s^*_t - \s^*_{t-1}} \leq O(\frac{\eta^2}{\lambda^2})\normsq{\v_t} + O(\frac{1}{\lambda^4})\]
% the second term would finally results in a $O(\frac{1}{\gamma^2\lambda^2})$ term in the desired bound, which \emph{cannot} be at the level of $\epsilon^2$.
% % \textbf{Possible direction}: We need to build convergence of $\s^*\to 0$ when $\w\to\w^*$.
% Below we show a stronger version of lemma A.7 in~\cite{kwon2024complexity}:

\subsection{Lemma for Theorem~\ref{thm:mss}}
Below we provide better bound with the extra Assumption~\ref{asm:mss} for recursion of $V_t$ and $S_t$.\\
Note that the proof given the extra mean-square smoothness assumption is quite similar to the corresponding lemma in the previous section, we only highlight the difference instead of provide the complete proof.

\begin{lem}[Recursive bound for momentum with MSS]\label{lem:mmt_recur_mss}    
Under Assumption~\ref{asm:func},~\ref{asm:oracle} and~\ref{asm:mss}, for $\beta\leq 1$ it holds that
\begin{align}\label{eq:mmt_err_mss}
    \E V_{t+1} \leq (1-\beta)\E V_t + \frac{2L_*^2\eta^2}{\beta}\E\normsq{\v_{t+1}} + 4\beta\E\normsq{G^w_{t+1} - \nabla\L_\lambda^*(\w_{t+1})} + \frac{\beta^2\tilde\sigma^2}{B}
\end{align}
where $G_{t+1}^w = \E_{(t+1)+} \hat{G}_{t+1}^w$ and $\tilde\sigma^2 = 2\sigma_f^2 + 8\frac{\tilde L_g^2C_f^2}{\mu_g^2} = O(\bar{\ell}^2\kappa^2)$. 
% From lemma 3.2 in~\cite{kwon2023fully} we have that $\y^*_\lambda(\w_t)$ is $C_{y*} = \frac{2(L_f+\lambda L_g)}{\lambda \mu_g} = O(\kappa)$-lipschitz, we can handle the contraction of the 1st term in RHS in a classic way while the most difficult part is the second term.
\end{lem}
\begin{proof}
The proof is almost the same as that of Lemma~\ref{lem:mmt_recur} except using~(\ref{eq:w_var_mss}) instead of~(\ref{eq:w_var}) to bound the variance term of $\hat{G}_t^w$.    
\end{proof}

% \begin{lem}[Descent lemma for $\y$ with MSS]    
% Denote $Y_t = \normsq{\y_{t+1} - \y^*_{\lambda,t}}$. If $\gamma \leq \frac{\mu_y}{L_y^2}$, we have 
% \begin{align}\label{eq:y_desc_mss}
%     \E Y_t \leq (1-\frac{\gamma\mu_y}{2})\E Y_{t-1} + \frac{2C_{y*}^2\eta^2}{\gamma\mu_y}\E\normsq{\v_{t}} + \frac{\gamma^2\sigma_y^2}{B}
% \end{align}
% where $\sigma_y^2 = \frac{2\sigma_f^2}{\lambda^2} + 2\sigma_g^2 = O(1)$ and $C_{y*} = \frac{2(L_f+\lambda L_g)}{\lambda \mu_g} = O(\kappa)$
% \end{lem}
% We do not provide a separate descent lemma for $\y$ here since MSS assumption won't change the way we build recursive bound for error of $\y$.

\begin{lem}[Descent lemma for $\s$ with MSS]   
Suppose that Assumption~\ref{asm:func},~\ref{asm:oracle} and~\ref{asm:mss} hold. Let $\gamma \leq \min\{\frac{\mu_g}{4L_g^2}, \frac{1}{2\mu_g}\}$, it holds that
\begin{align}\label{eq:s_desc_pre_mss}
    \E S_t &\leq \E\normsq{\bar{\s}_t - \s_t^*}\nonumber\\
    &\leq \left(1-\frac{\gamma\mu_g}{4}\right)\E S_{t-1} + \frac{4}{\gamma\mu_g}\E\normsq{\s^*_t - \s^*_{t-1}} + \frac{8\gamma}{\mu_g}\left(\frac{L_f^2}{\lambda^2}+Q_g^2r_\lambda^2\right)\left(\E Y_{t-1} + \eta^2C_{y*}^2\E\normsq{\v_t} \right) \nonumber\\&\qquad+ \frac{2\gamma^2}{B\lambda^2}(\sigma_f^2 + \frac{4C_f^2\tilde L_g^2}{\mu_g^2}).
\end{align}
\end{lem}
\begin{proof}
    The main difference to the proof of Lemma~\ref{lem:s_desc} is that we leverage a better bound for the variance term of $\hat{G}_t^s$ given the MSS assumption:
\begin{align}
    \text{var}(\hat{G}_t^s) &= \E_t\normsq{\hat{G}_t^s - G_t^s}\nonumber\\
    &\leq \frac{1}{B}\left(\frac{2}{\lambda^2}\E_t\normsq{\nabla_2 f(\w_t, \y_t, \xi^t_{fy}) - \nabla_2 f(\w_t, \y_t)}\right.\nonumber\\&\qquad + 4\E_t\normsq{\nabla_2 g(\w_t, \y_t, \xi^t_{gy}) - \nabla_2 g(\w_t, \z_t, \xi^t_{gy})}\nonumber\\ 
    &\qquad\left. + 4\normsq{\nabla_2 g(\w_t, \y_t) - \nabla_2 g(\w_t, \z_t)}\right)\nonumber\\
    &\leq \frac{1}{B}\left(\frac{2\sigma_f^2}{\lambda^2} + 8\tilde L_g^2\normsq{\y_t - \z_t}\right)\nonumber\\
    &\leq \frac{1}{B}\left(\frac{2\sigma_f^2}{\lambda^2} + 8 \frac{C_f^2\tilde L_g^2}{\mu_g^2\lambda^2}\right)
\end{align}
where $\xi^t_{fy}, \xi^t_{gy}$ are randomly sample from $\B^t_{fy}, \B^t_{gy}$. For the following derivation, we just need to follow inequality~(\ref{eq:s_pre_derive}) but replace with the above bound for the variance term.
\end{proof}
\section{Proof of Theorem~\ref{thm:std}}
\label{sec:std}
\begin{proof}[Proof of Theorem~\ref{thm:std}]     
% We summarize important bound we obtain so far, which are derived from~\ref{eq:noncvx_starter}, 
By plugging~(\ref{eq:projed_grad_est_unpack}) into ~(\ref{eq:mmt_err_bi_penal}) and plugging~(\ref{eq:s_lip}) into~(\ref{eq:s_desc_pre}) we have:
% \begin{align}
%     \E_t \L_\lambda^*(\w_{t+1}) \leq \L_\lambda^*(\w_t) + \frac{\eta}{2}\normsq{\v_{t+1} - \nabla\L_\lambda^*(\w_t)} - \frac{\eta}{4}\normsq{\v_{t+1}} - \frac{\eta}{2}\normsq{\nabla\L_\lambda^*(\w_t)},
% \end{align}
\begin{align}\label{eq:mmt_err_recur}
    \E V_{t+1} &\leq (1-\beta)\E V_t + \frac{2L_*^2\eta^2}{\beta}\E\normsq{\v_{t+1}} + 12\beta L_{y*}^2\E\normsq{\y_{t+2} - \y^*_\lambda(\w_{t+1})}\nonumber\\
    &\qquad+ 12\beta\lambda^2 L_g^2\E\normsq{\s_{t+2} - \s^*(\w_{t+1})} + \frac{12\beta C_f^2L_{y*}^2}{\mu_g^2\lambda^2} + \frac{\beta^2\lambda^2\sigma^2}{B}\nonumber\\
    &=(1-\beta)\E V_t + \frac{2L_*^2\eta^2}{\beta}\E\normsq{\v_{t+1}} + 12\beta L_{y*}^2\E Y_{t+1}+ 12\beta\lambda^2 L_g^2\E S_{t+1} + \frac{12\beta C_f^2L_{y*}^2}{\mu_g^2\lambda^2} + \frac{\beta^2\lambda^2\sigma^2}{B},
\end{align}
% \begin{align}\label{eq:y_desc}
%     \E Y_t \leq (1-\frac{\gamma\mu_y}{2})Y_{t-1} + \frac{2C_{y*}^2\eta^2}{\gamma\mu_y}\normsq{\v_t} + \frac{\gamma^2\sigma_y^2}{B},
% \end{align}
\begin{align}\label{eq:s_desc}
    \E S_t &\leq \E\normsq{\bar{\s}_t - \s_t^*}\nonumber\\
    &\leq \left(1-\frac{\gamma\mu_g}{4}\right)\E S_{t-1} + \frac{4C_s^2\eta^2}{\lambda^2\gamma\mu_g}\E\normsq{\v_t} + \frac{8\gamma}{\mu_g}\left(\frac{L_f^2}{\lambda^2}+Q_g^2r_\lambda^2\right)\left(\E Y_{t-1} + \eta^2C_{y*}^2\E\normsq{\v_t} \right)\nonumber\\&\qquad + \frac{2\gamma^2}{B}(\sigma_y^2 + \sigma_g^2).
\end{align}

Combining~(\ref{eq:noncvx_starter}), ~(\ref{eq:mmt_err_recur}), ~(\ref{eq:y_desc}) and~(\ref{eq:s_desc}), then summing over $t=0, 1, \cdots, T-1,$ we have
\begin{align}\label{eq:bi_penal_main_bound}
    \frac{1}{T}\sum_{t=0}^{T-1}\E\normsq{\nabla\L_\lambda^*(\w_t)}&\leq \frac{2(\L_\lambda^*(\w_0) - \inf_\w\L_\lambda^*(\w))}{\eta T} + \frac{\E V_0}{\beta T}\nonumber\\&\qquad + \frac{24\left((L_{y*}^2+32\kappa^2(L_f^2 + \frac{Q_g^2C_f^2}{\mu_g^2})) \E Y_1+2\lambda^2L_g^2\E S_1\right)}{\gamma\mu_g T}\nonumber\\
    & - \left(\frac{1}{2} - \frac{2L_*^2\eta^2}{\beta^2} - \frac{48\eta^2\left(C_{y*}^2(L_{y*}^2 + 32\kappa^2(L_f^2 + \frac{Q_g^2C_f^2}{\mu_g^2})) + 4L_g^2C_s^2\right)}{\gamma^2\mu_g^2}\right. \nonumber\\ &\qquad \left. - 384\eta^2\kappa^2C_{y*}^2(L_f^2+\frac{Q_g^2C_f^2}{\mu_g^2})\right)\frac{1}{T}\sum_{t=0}^{T-1}\E\normsq{\v_{t+1}}\nonumber\\
    & + \frac{\beta \lambda^2\sigma^2}{B} + \frac{24\gamma\sigma_y^2}{\mu_g B}\left(L_{y*}^2+32\kappa^2(L_f^2 + \frac{Q_g^2C_f^2}{\mu_g^2})\right) + \frac{12C_f^2L_{y*}^2}{\mu_g^2\lambda^2}\nonumber\\&\qquad + \frac{96\gamma\lambda^2L_g^2}{\mu_g B}(\sigma_y^2 + \sigma_g^2).
\end{align}
% To make RHS$\leq \epsilon^2$, we require that $\beta\sigma^2\leq O(\epsilon^2), \gamma\lambda^2\leq O(\epsilon^2), \eta\leq O(\min\{\beta, \gamma\}), T\geq \Omega(\frac{1}{\eta\epsilon^2})$, where $\beta = \Theta(\gamma) = \Theta(\eta) = \Theta(\epsilon^4), T = O(\epsilon^{-6})$ satisfy all.
Denote $\Delta=\L_\lambda^*(\w_0) - \inf_\w\L_\lambda^*(\w)$. With Lemma~\ref{lem:penal_property} we have that
\[\Delta \leq F(\w_0) - \inf_\w F(\w) + O(\frac{\bar\ell\kappa^2}{\lambda})\leq O(1)\]
Note that $\E\normsq{\v_1} = \beta^2 \E\normsq{\hat{G}_0^w} = O(\beta^2\lambda^2)$, we have:
\[\E V_0\leq 2\E\normsq{\v_1} + 2\normsq{\nabla \L_\lambda^*(\w_0)}\leq O(\beta^2\lambda^2) + 4L_*\Delta\leq O(\beta^2\lambda^2+\kappa^3).\]
With eqn.~(\ref{eq:y_desc}) we have that
\[\E Y_1 \leq (1-\frac{\gamma\mu_y}{2}) Y_0 + \frac{2C_{y*}^2\eta^2}{\gamma\mu_y}\normsq{\v_1} + \frac{\gamma^2\sigma_y^2}{B}.\]
According the setting of $\beta,\gamma$ and $\lambda$ that we show later, the second and the last term will be dominated by the first term. Therefore, we have $\E Y_1 = O(Y_0)$. Besides, 
\[Y_0 = \normsq{\y_1 - \y_\lambda^*(\w_0)}=  \normsq{\y_0 - \gamma\hat{G}_t^y - \y_\lambda^*(\w_0)}\leq 2\normsq{\y_0 - \y_\lambda^*(\w_0)} + 2\gamma^2\normsq{\hat{G}_0^y}\leq O(\frac{1}{\kappa}), \]
where the last inequality holds from the warm up stage for $\y$. Combining the 2 inequalities above yields that $\E Y_1\leq O(\frac{1}{\kappa})$. In addition, $\E S_1\leq 2\E \normsq{\s_2} + 2\E\normsq{\s^*(\w_1)}\leq O(r_\lambda^2) = O(\frac{\kappa^2}{\lambda^2})$. 
% \begin{align}
%     Y_0 &= \normsq{\y_1 - \y_\lambda^*(\w_0)}\nonumber\\
%     &=  \normsq{\y_0 - \gamma\hat{G}_t^y - \y_\lambda^*(\w_0)}\nonumber\\
%     &\leq 2\normsq{\y_0 - \y_\lambda^*(\w_0)} + 2\gamma^2\normsq{\hat{G}_0^y}\nonumber\\
%     &\leq O(\frac{1}{\kappa})    
% \end{align}

To make RHS of eqn.~(\ref{eq:bi_penal_main_bound})$\leq \epsilon^2$ with constant level batchsize $B=O(1)$, we require the following condition holds:
\begin{align}
    \beta\lambda^2\sigma^2 \leq O(\epsilon^2), \gamma\sigma_y^2\bar{\ell}\kappa^5\leq O(\epsilon^2), \frac{\bar{\ell}^2\kappa^4}{\lambda^2}\leq O(\epsilon^2), \bar{\ell}\kappa\gamma\lambda^2(\sigma_y^2+\sigma_g^2)\leq O(\epsilon^2),\\
    \max\left\{\frac{L_*^2\eta^2}{\beta^2}, \frac{\eta^2O(\kappa^8)}{\gamma^2}, \eta^2\bar{\ell}^2\kappa^6\right\} \leq O(1),\\
    \frac{\Delta}{\eta T}\leq O(\epsilon^2), \frac{V_0}{\beta T}\leq O(\epsilon^2), \frac{O(\bar{\ell}\kappa^5 Y_1 + \lambda^2\bar{\ell}\kappa S_1)}{\gamma T}\leq O(\epsilon^2).
\end{align}
Therefore, by setting $\lambda = \Theta(\frac{\kappa^3}{\epsilon}), \beta = \Theta(\frac{\epsilon^4}{\kappa^6}), \gamma = \Theta(\frac{\epsilon^4}{\kappa^7}), \eta = \Theta(\frac{\epsilon^4}{\kappa^{11}})$, it's sufficient for Algorithm~\ref{alg:s2f2bo} to return an $O(\epsilon)$-stationary point of $\L_\lambda^*(\w)$ within $T = O(\frac{\kappa^{11}}{\epsilon^6})$ iterations. Then according to Lemma~\ref{lem:penal_property}, it is also an $O(\epsilon)$-stationary point of $F(\w)$. %Besides, if we add an extra "warm-up" stage before starting Algorithm~\ref{alg:s2f2bo} to get a better starting point of $\y$, i.e. $Y_1 = O(1/\kappa)$ which can be satisfied by a simple SGD with $\tilde O(\kappa)$-iteration due to the strong convexity of the sub-problem of $\y$, we can improve the dependence of total complexity on $\kappa$ to be $T = O(\frac{\kappa^{11}}{\epsilon^6})$.
\end{proof}

\section{Proof of Theorem~\ref{thm:mss}}\label{sec:mss}
The proof of Theorem~\ref{thm:mss} is quite similar to that of Theorem~\ref{thm:std} except that we need to leverage better recursion lemma of $\v$ and $\s$. 
\begin{proof}[Proof of Theorem~\ref{thm:mss}]     
By plugging~(\ref{eq:projed_grad_est_unpack}) into ~(\ref{eq:mmt_err_mss}) and plugging~(\ref{eq:s_lip}) into~(\ref{eq:s_desc_pre_mss}), we have
% \begin{align}
%     \E_t \L_\lambda^*(\w_{t+1}) \leq \L_\lambda^*(\w_t) + \frac{\eta}{2}\normsq{\v_{t+1} - \nabla\L_\lambda^*(\w_t)} - \frac{\eta}{4}\normsq{\v_{t+1}} - \frac{\eta}{2}\normsq{\nabla\L_\lambda^*(\w_t)}
% \end{align}
\begin{align}\label{eq:mmt_err_recur_mss}
    \E V_{t+1} &\leq (1-\beta)\E V_t + \frac{2L_*^2\eta^2}{\beta}\E\normsq{\v_{t+1}} + 12\beta L_{y*}^2\E\normsq{\y_{t+2} - \y^*_\lambda(\w_{t+1})}\nonumber\\
    &\qquad+ 12\beta\lambda^2 L_g^2\E\normsq{\s_{t+2} - \s^*(\w_{t+1})} + \frac{12\beta C_f^2L_{y*}^2}{\mu_g^2\lambda^2} + \frac{\beta^2\tilde\sigma^2}{B}\nonumber\\
    &=(1-\beta)\E V_t + \frac{2L_*^2\eta^2}{\beta}\E\normsq{\v_{t+1}} + 12\beta L_{y*}^2\E Y_{t+1}+ 12\beta\lambda^2 L_g^2\E S_{t+1} + \frac{12\beta C_f^2L_{y*}^2}{\mu_g^2\lambda^2} + \frac{\beta^2\tilde\sigma^2}{B},
\end{align}
% \begin{align}\label{eq:y_desc}
%     \E Y_t \leq (1-\frac{\gamma\mu_y}{2})Y_{t-1} + \frac{2C_{y*}^2\eta^2}{\gamma\mu_y}\normsq{\v_t} + \frac{\gamma^2\sigma_y^2}{B}
% \end{align}
\begin{align}\label{eq:s_desc_mss}
    \E_t S_t &\leq \E_t\normsq{\bar{\s}_t - \s_t^*}\nonumber\\
    &\leq \left(1-\frac{\gamma\mu_g}{4}\right)S_{t-1} + \frac{4C_s^2\eta^2}{\lambda^2\gamma\mu_g}\normsq{\v_t} + \frac{8\gamma}{\mu_g}\left(\frac{L_f^2}{\lambda^2}+Q_g^2r_\lambda^2\right)\left(Y_{t-1} + \eta^2C_{y*}^2\normsq{\v_t} \right)\nonumber\\&\qquad + \frac{2\gamma^2}{B\lambda^2}(\sigma_f^2 + \frac{4C_f^2\tilde L_g^2}{\mu_g^2}).
\end{align}

Combine~(\ref{eq:noncvx_starter}),~(\ref{eq:mmt_err_recur_mss}),~(\ref{eq:y_desc}) and~(\ref{eq:s_desc_mss}), then sum over $t=0, 1, \cdots, T-1$, we have
\begin{align}\label{eq:mss_main_bound}
    \frac{1}{T}\sum_{t=0}^{T-1}\E\normsq{\nabla\L_\lambda^*(\w_t)}&\leq \frac{2(\L_\lambda^*(\w_0) - \inf_\w\L_\lambda^*(\w))}{\eta T} + \frac{\E V_0}{\beta T}\nonumber\\
    &\qquad + \frac{24\left((L_{y*}^2+32\kappa^2(L_f^2 + \frac{Q_g^2C_f^2}{\mu_g^2})) \E Y_1+2\lambda^2L_g^2\E S_1\right)}{\gamma\mu_g T}\nonumber\\
    & - \left(\frac{1}{2} - \frac{2L_*^2\eta^2}{\beta^2} - \frac{48\eta^2\left(C_{y*}^2(L_{y*}^2 + 32\kappa^2(L_f^2 + \frac{Q_g^2C_f^2}{\mu_g^2})) + 4L_g^2C_s^2\right)}{\gamma^2\mu_g^2}\right.\nonumber\\
    &\left.\qquad- 384\eta^2\kappa^2C_{y*}^2(L_f^2+\frac{Q_g^2C_f^2}{\mu_g^2})\right)\frac{1}{T}\sum_{t=0}^{T-1}\E\normsq{\v_{t+1}}\nonumber\\
    & + \frac{\beta \tilde\sigma^2}{B} + \frac{24\gamma\sigma_y^2}{\mu_g B}\left(L_{y*}^2+32\kappa^2(L_f^2 + \frac{Q_g^2C_f^2}{\mu_g^2})\right) + \frac{12C_f^2L_{y*}^2}{\mu_g^2\lambda^2}\nonumber\\&\qquad + \frac{96\gamma L_g^2}{\mu_g B}(\sigma_f^2 + \frac{4C_f^2\tilde L_g^2}{\mu_g^2}).
\end{align}
% To make RHS$\leq \epsilon^2$, we require that $\beta\sigma^2\leq O(\epsilon^2), \gamma\lambda^2\leq O(\epsilon^2), \eta\leq O(\min\{\beta, \gamma\}), T\geq \Omega(\frac{1}{\eta\epsilon^2})$, where $\beta = \Theta(\gamma) = \Theta(\eta) = \Theta(\epsilon^4), T = O(\epsilon^{-6})$ satisfy all.
As shown in proof of Theorem~\ref{thm:std}, we have that:
\[\L_\lambda^*(\w_0) - \inf_\w\L_\lambda^*(\w)\leq O(1), \E V_0\leq O(\kappa^3), \E Y_1\leq O(\frac{1}{\kappa}), \E S_1\leq \frac{\kappa^2}{\lambda^2}\]
Therefore, to make RHS of eqn.~(\ref{eq:mss_main_bound})$\leq \epsilon^2$ with constant level batchsize $B=O(1)$, and note that $\tilde\sigma^2 = O(\bar{\ell}^2\kappa^2)$, we require the following condition holds:
\begin{align}
    \beta\bar{\ell}^2\kappa^2 \leq O(\epsilon^2), \gamma\bar{\ell}\kappa^5\leq O(\epsilon^2), \frac{\bar{\ell}^2\kappa^4}{\lambda^2}\leq O(\epsilon^2), \bar{\ell}^3\kappa^3\gamma\leq O(\epsilon^2),\\
    \max\left\{\frac{L_*^2\eta^2}{\beta^2}, \frac{\eta^2O(\kappa^8)}{\gamma^2}, \eta^2\bar{\ell}^2\kappa^6\right\} \leq O(1),\\
    \frac{\Delta}{\eta T}\leq O(\epsilon^2), \frac{V_0}{\beta T}\leq O(\epsilon^2), \frac{O(\bar{\ell}\kappa^5 Y_1 + \lambda^2\bar{\ell}\kappa S_1)}{\gamma T}\leq O(\epsilon^2).
\end{align}
% Therefore, if we set $\lambda = \Theta(\frac{\kappa^3}{\epsilon}), \beta = \Theta(\frac{\epsilon^2}{\kappa^2}), \gamma = \Theta(\frac{\epsilon^2}{\kappa^5}), \eta = \Theta(\frac{\epsilon^2}{\kappa^{9}})$, we need $T = O(\frac{\kappa^{10}}{\epsilon^4})$ to reach $O(\epsilon)$-stationary point of $\L_\lambda^*(\w)$, which, according to ??, leads to $O(\epsilon)$-stationary point of $F(\w)$. Besides, if we add an extra "warm-up" stage before starting Algorithm~\ref{alg:s2f2bo} to get a better starting point of $\y$, i.e. $Y_1 = O(1/\kappa)$, we can improve the dependence of total complexity on $\kappa$ to be $T = O(\frac{\kappa^{9}}{\epsilon^4})$.

Therefore, by setting $\lambda = \Theta(\frac{\kappa^3}{\epsilon}), \beta = \Theta(\frac{\epsilon^2}{\kappa^2}), \gamma = \Theta(\frac{\epsilon^2}{\kappa^5}), \eta = \Theta(\frac{\epsilon^2}{\kappa^{9}})$, it's sufficient for Algorithm~\ref{alg:s2f2bo} to return an $O(\epsilon)$-stationary point of $\L_\lambda^*(\w)$ within $T = O(\frac{\kappa^{9}}{\epsilon^4})$ iterations. Then according to Lemma~\ref{lem:penal_property}, it is also an $O(\epsilon)$-stationary point of $F(\w)$. 
\end{proof}

\section{Counter example of inequality~\ref{eq:wrong_crossterm_bound}}\label{sec:counter_eg}

We provide a counterexample showing that even under Assumptions~\ref{asm:func}--~\ref{asm:mss},
the inequality
\[
    \left\langle G_t^s, s_t-s_t^* \right\rangle
    \ge c\|s_t-s_t^*\|^2
\]
does not necessarily hold for any positive constant $c$.
Consider a scalar lower-level variable $y\in\mathbb{R}$ and define
\[
    f(w,y)=\cos y-1,
    \qquad
    g(w,y)=\frac{1}{2}y^2,
\]
where both functions are independent of $w$.
These functions are infinitely differentiable, $f$ is globally
$1$-Lipschitz and $1$-smooth, and $g$ is $1$-smooth and
$1$-strongly convex in $y$, with a constant Hessian.
For any $\lambda\ge 2$, the penalized lower-level objective satisfies
\[
    \frac{\partial^2}{\partial y^2}
    \left(\lambda^{-1}f(w,y)+g(w,y)\right)
    =1-\lambda^{-1}\cos y
    \ge \frac{1}{2}.
\]
Since its first derivative vanishes at $y=0$, we have
\[
    y_\lambda^*(w)=y^*(w)=0,
    \qquad
    s^*(w)=0,
    \qquad
    F(w)=0.
\]
Thus, Assumption~1 holds. Moreover, taking exact gradients as the
stochastic oracles satisfies Assumption~\ref{asm:oracle} with zero variance and
Assumption~\ref{asm:mss} with $\widetilde L_g=1$.

Now choose any $y_t\in(0,1]$ and set
\[
    z_t=y_t-\frac{\sin y_t}{\lambda}.
\]
Then
\[
    s_t=y_t-z_t=\frac{\sin y_t}{\lambda}>0,
    \qquad
    |s_t|\le\frac{1}{\lambda}
    =\frac{C_f}{\mu_g\lambda}=r_\lambda,
\]
so these iterates satisfy the projected-coupling constraint.
Nevertheless, the corresponding gap-update direction is
\[
\begin{aligned}
    G_t^s
    &=
    \lambda^{-1}\nabla_2 f(w_t,y_t)
    +\nabla_2 g(w_t,y_t)
    -\nabla_2 g(w_t,z_t)\\
    &=
    -\frac{\sin y_t}{\lambda}+y_t-z_t
    =0.
\end{aligned}
\]
Consequently, for every $c>0$,
\[
    \left\langle G_t^s,s_t-s_t^*\right\rangle
    =0
    <
    c\,\frac{\sin^2 y_t}{\lambda^2}
    =
    c\|s_t-s_t^*\|^2,
\]
which contradicts the proposed inequality.

\section{More experiment detail}\label{sec:exp_detail}
\subsection{Hyper-parameter Tuning}
We use linear schedule for $\lambda$ of F2SA/F3SA which start from $\lambda_0$ and tune the final $\lambda$; We further tune $\beta$ for our method SICO in $\{0.01, 0.1, 0.3, 0.5, 1\}$. For stoc-SGHA we further tune $\tau$ in $\{10^{-4}, 10^{-3}\}$; $p_w$ in $\{0.005, 0.01, 0.02\}$; $p_\lambda$ in $\{0.00001, 0.0001\}$ and $\beta$ in $\{0.16, 0.24, 0.32\}$; we fix $C_\lambda=10$.

\subsection{Additional experiment result}\label{subsec:extra_exp_fig}
In this subsection, we report in figure~\ref{fig:data_hpc_clean}-~\ref{fig:data_hpc_testacc} for garbage weight, clean weight, validation loss and test accuracy for the data hyper-cleaning task on MNIST and fashion-MNIST dataset.
\begin{figure}[t]
\begin{center}
\includegraphics[scale=0.35]{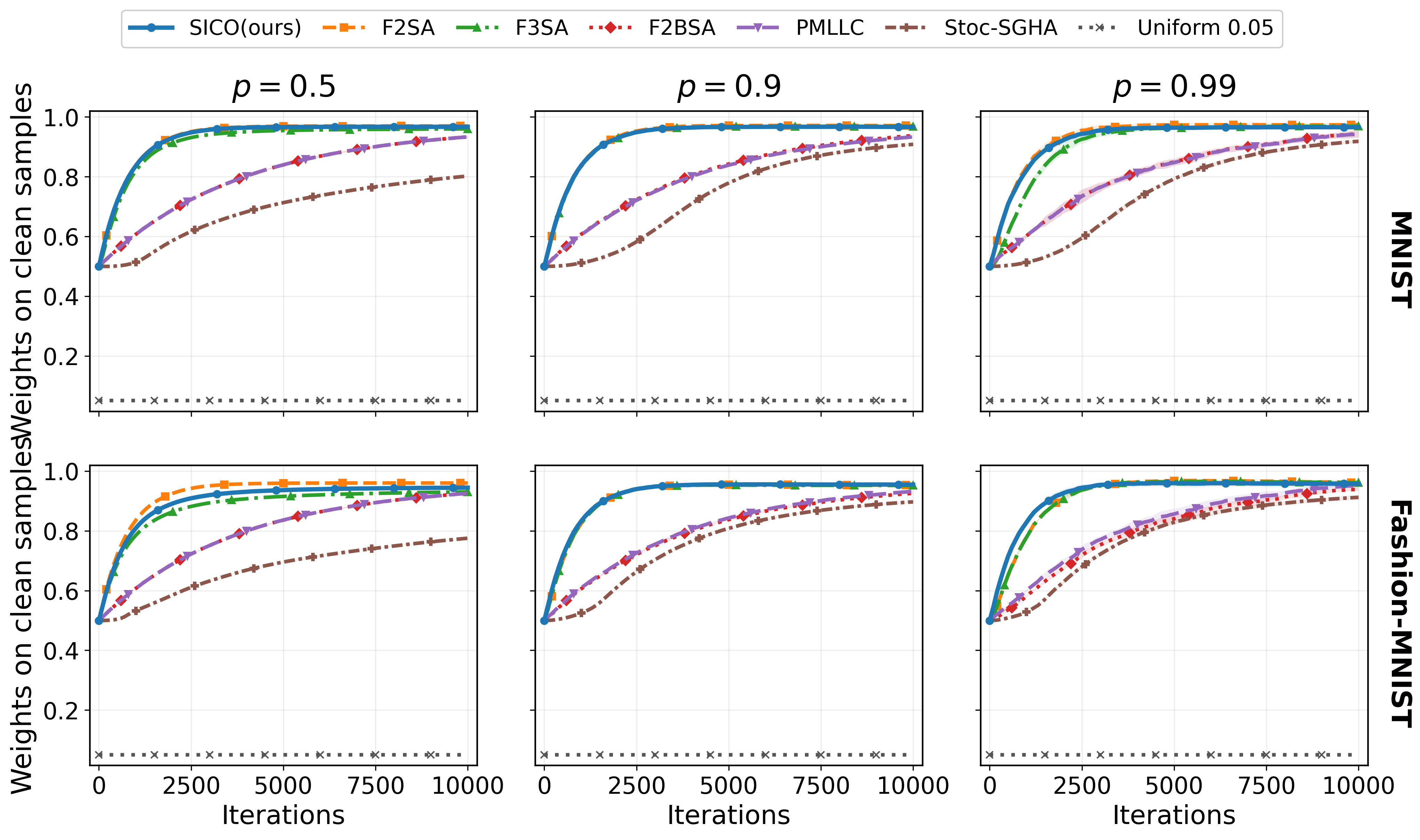}
\end{center}
\caption{Clean weight comparison of first order method of NC-SC SBO problem for data hyper-cleaning task. Upper row: MNIST dataset; Lower row: Fashion-MNIST dataset.}\label{fig:data_hpc_clean}
\end{figure}

\begin{figure}[t]
\begin{center}
\includegraphics[scale=0.35]{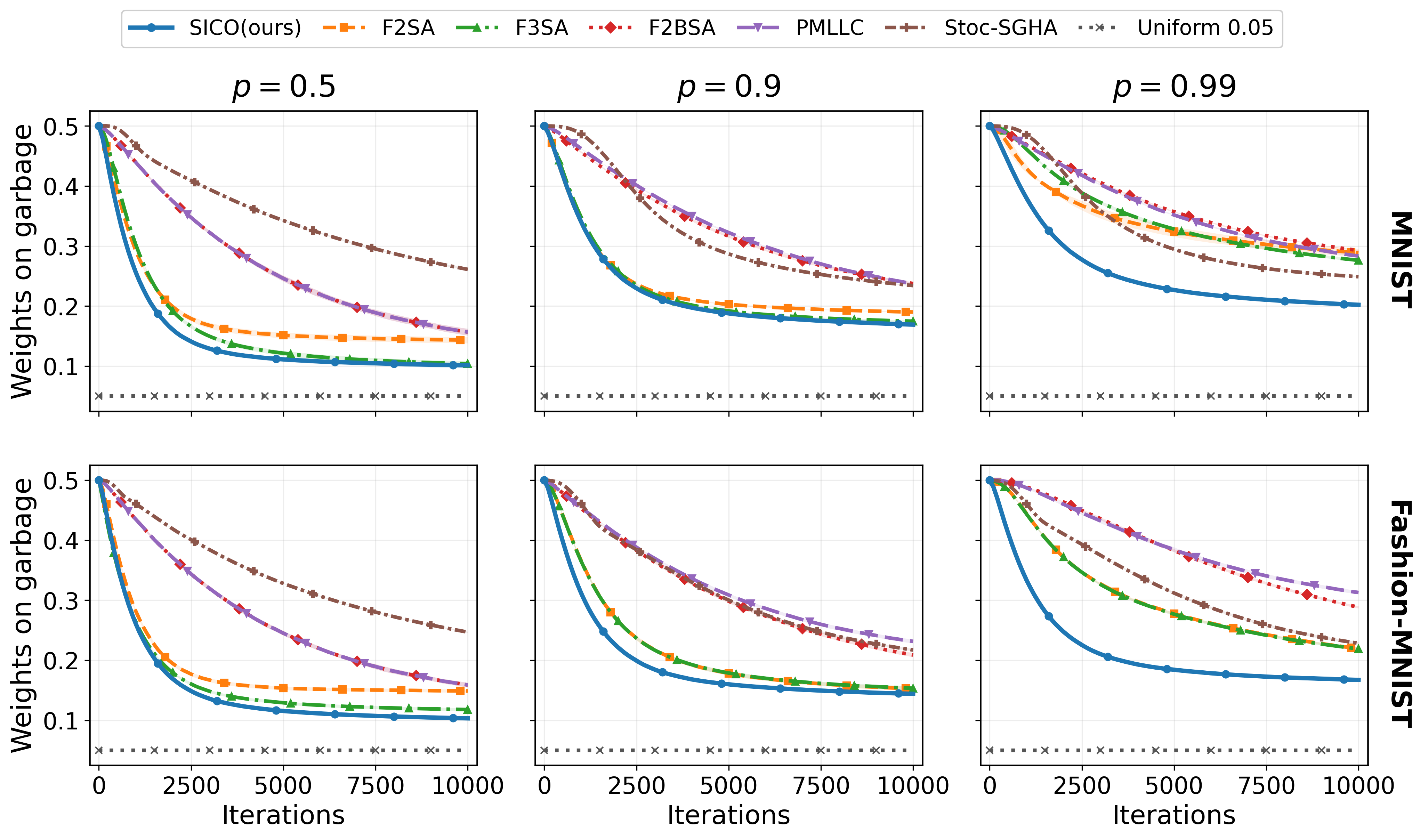}
\end{center}
\caption{Garbage weight comparison of first order method of NC-SC SBO problem for data hyper-cleaning task. Upper row: MNIST dataset; Lower row: Fashion-MNIST dataset.}\label{fig:data_hpc_garbage}
\end{figure}

\begin{figure}[t]
\begin{center}
\includegraphics[scale=0.35]{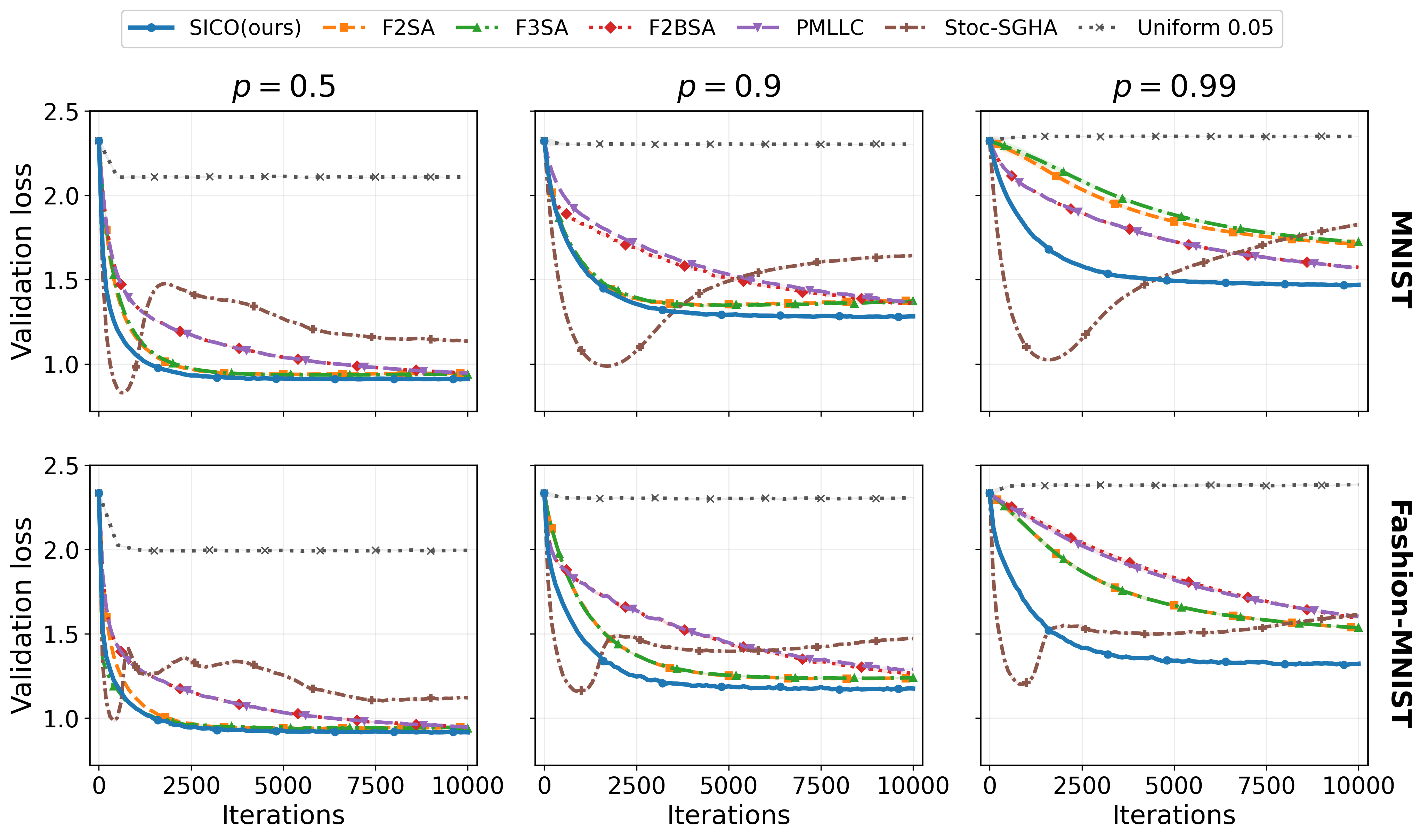}
\end{center}
\caption{Validation loss comparison of first order method of NC-SC SBO problem for data hyper-cleaning task. Upper row: MNIST dataset; Lower row: Fashion-MNIST dataset.}\label{fig:data_hpc_valloss}
\end{figure}

\begin{figure}[t]
\begin{center}
\includegraphics[scale=0.35]{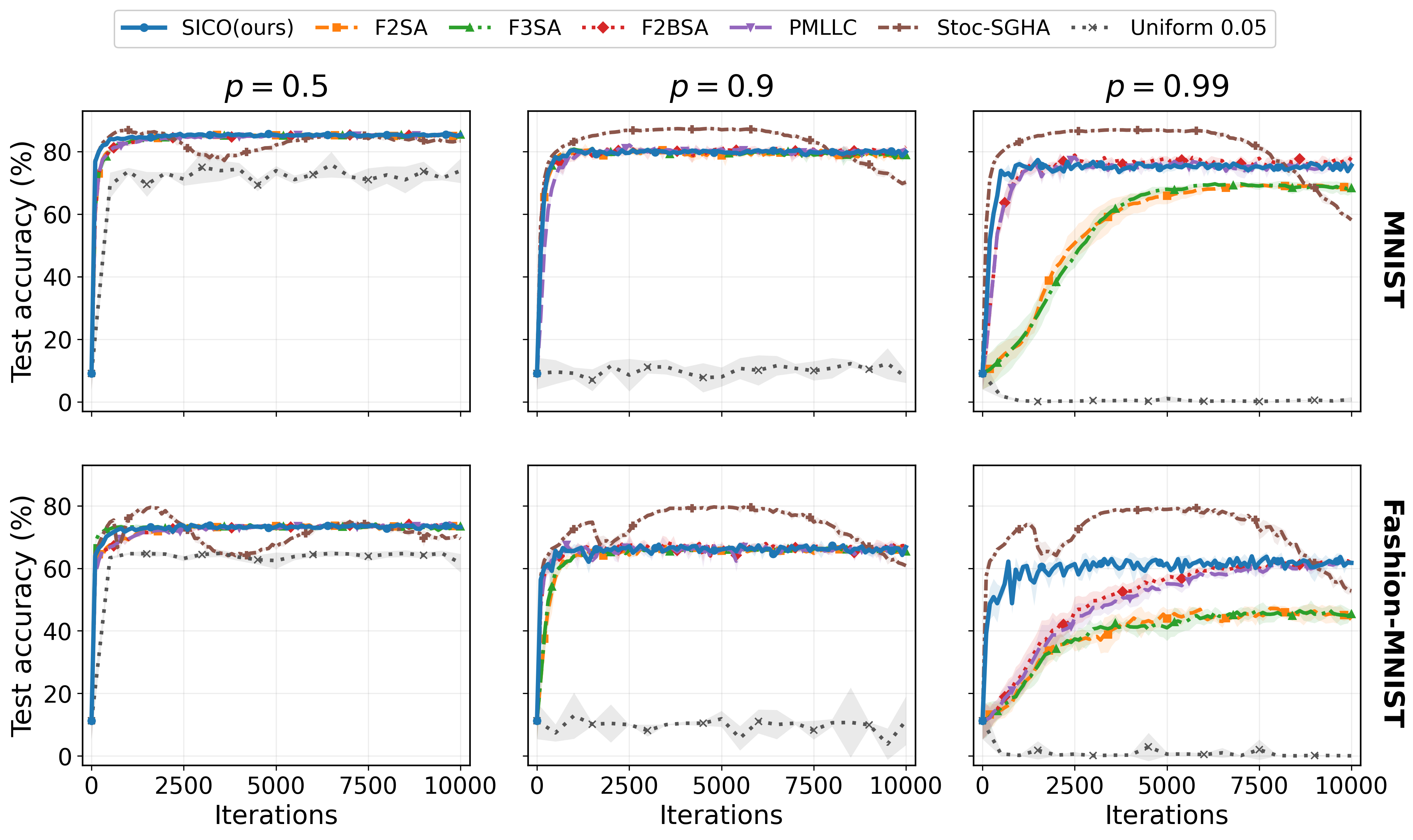}
\end{center}
\caption{Test accuracy comparison of first order method of NC-SC SBO problem for data hyper-cleaning task. Upper row: MNIST dataset; Lower row: Fashion-MNIST dataset.}\label{fig:data_hpc_testacc}
\end{figure}
\end{document}